\documentclass[journal]{IEEEtran}
\usepackage{cite}
\usepackage{amsmath,amssymb,amsfonts}
\usepackage{algorithmic}
\usepackage{graphicx}
\usepackage{textcomp}
\usepackage{color}
\usepackage{bm}
\usepackage{amsthm}
\newtheorem{assumption}{$Assumption$}
\usepackage{algorithm}
\usepackage{booktabs}

\ifCLASSINFOpdf
\else
\fi
\begin{document}
%
\title{
	Resource-aware Exact Decentralized Optimization Using Event-triggered Broadcasting
	}
%
%
%

\author{Changxin Liu, \IEEEmembership{Student Member, IEEE,} 
	Huiping Li, \IEEEmembership{Member, IEEE,} 
    and Yang Shi, \IEEEmembership{Fellow, IEEE}

	\thanks{C. Liu and Y. Shi are with the Department of Mechanical Engineering, University of Victoria, Victoria, BC V8W 3P6, Canada (e-mail: chxliu@uvic.ca; yshi@uvic.ca). }
	\thanks{H. Li is with the School of Marine Science and Technology, Northwestern Polytechnical University, Xi'an, 710072, China (e-mail: lihuiping@nwpu.edu.cn).}}

\markboth{Journal of \LaTeX\ Class Files,~Vol.~14, No.~8, August~2015}%
{Shell \MakeLowercase{\textit{et al.}}: Bare Demo of IEEEtran.cls for IEEE Journals}
%



\maketitle

\begin{abstract}
This work addresses the decentralized optimization problem where a group of agents with coupled private objective functions work together to exactly optimize the summation of local interests. 
Upon modeling the decentralized problem as an equality-constrained centralized one, 
we leverage the linearized augmented Lagrangian method (LALM) to design an event-triggered decentralized algorithm that only requires light local computation at generic time instants and peer-to-peer communication at sporadic triggering time instants.  
The triggering time instants for each agent are locally determined by comparing the deviation between true and broadcast primal variables with certain triggering thresholds.
{ 
	Provided that the threshold is summable over time, we established a new upper bound for the effect of triggering behavior on the primal-dual residual. Based on this, the same convergence rate $O(\frac{1}{k})$ with periodic algorithms is secured for nonsmooth convex problems. Stronger convergence results have been further established for strongly convex and smooth problems, that is, the iterates linearly converge with exponentially decaying triggering thresholds.}
We examine the developed strategy in two common optimization problems; comparison results illustrate its performance and superiority in exploiting communication resources. 
\end{abstract}


%
\IEEEpeerreviewmaketitle

\section{Introduction}
%
%
%
%
%
%
Decentralized optimization methods have received increasing attention recently due to { their} key role in advancing future developments of many engineering areas as diverse as wireless decentralized control systems, sensor networks, and decentralized machine learning \cite{nedicnetworktopology,yangasurvey}. Usually, they involve a group of computing units that are connected via a communication network and rely on only local computation and peer-to-peer communication to cooperatively solve a large-scale optimization problem, where some coupling sources in the objective or/and constraint make the partition nontrivial.

This paper considers the case with coupled objective functions, i.e., the global objective is the sum of multiple private ones. In the literature, many efforts have been devoted to problems of this type \cite{yangasurvey}. 
{ Some methods \cite{duchidualaveraging,liudistributed} manage to agree on the estimate of the global gradient, and minimize the sum of a linear function characterized by this estimated gradient and a prox-function at each iteration to generate a sequence of local decision variables \cite{duchidualaveraging,liudistributed}.  
In marked contrast, the strategies in \cite{nedicconstrainedconsensus,nedicachieving,yuanontheconvergence,xuconvergence,quharnessing,shiextra,shionthelinear,xuabregman,fazlyabdistributed,seidmanachebyshev} directly seek consensus on the decision variable by iteratively shifting its local estimate about the global minimizer in light of the local (sub)gradient and the information from its immediate neighbors.} The convergence properties have been thoroughly investigated for this scheme with decaying and constant stepsizes in \cite{nedicconstrainedconsensus} and \cite{yuanontheconvergence}, respectively. It is worth mentioning that when using constant stepsizes with these methods, between the accumulation point and the global minimum, there is always an undesired gap whose magnitude is proportional to the stepsize. To achieve exact convergence, the authors in \cite{shiextra} further added a cumulative correction term to the iteration rule of decentralized gradient descent (DGD) \cite{yuanontheconvergence}; a variant of this method is reported in \cite{shiapro} to handle constraints.  Note that there are other interpretations for the method in \cite{shiextra}; therein please find more details. Another remedy to this problem was reported in \cite{xuconvergence,quharnessing,nedicachieving} where the local gradient used in DGD is replaced by an estimate of the global gradient supplied by the dynamic average consensus scheme. Although these methods share a similar iteration rule, the analyses are significantly different from one to another due to different network configurations. 
{ Another powerful design methodology for distributed optimization is to treat the decentralized optimization problem as a linear equality-constrained centralized one,
and use linearly constrained optimization paradigms.}
For example, the dual decomposition \cite{fazlyabdistributed}, the augmented Lagrangian method \cite{shiaugmented}, the alternating direction method of multipliers (ADMM) \cite{shionthelinear}, the Bregman method \cite{xuabregman}, and other primal-dual methods \cite{seidmanachebyshev,liaprimal,latafatnew,uribeadual} have been used to design decentralized algorithms. { It is worth to mention that the method in \cite{shiextra} and the gradient-tracking distributed optimization in \cite{quharnessing,nedicachieving} also have such primal-dual interpretations, as reported in \cite{liaprimal,jakoveticauni}.}

To reduce the communication cost of periodic algorithms, some asynchronous algorithms have been reported in the literature. For instance, the authors in \cite{nedicasynchronous} considered DGD with random communication link failures, and established convergence rate and error bound for decaying and constant stepsizes, respectively. Using a similar idea, reference \cite{jakovetic} presented an asynchronous DGD where only a randomized set of working agents choose to update their local variables. The authors proved that the local estimates converge to a neighborhood of the minimizer provided that the activation probability grows to $1$ asymptotically. The works \cite{weionthe,changasynchronous} considered asynchronous ADMM and established convergence rates. However, in these methods each agent is still dictated to exactly solving a subproblem at each iteration. Recently, reference \cite{wudecentralized} built an asynchronous decentralized consensus optimization algorithm based on \cite{shiextra} for a network of agents where communication delays may occur, and proved convergence. Another communication-efficient decentralized gradient method was reported in \cite{zhangdistributed}; its novelty may lie in the use of only signs of relative state information between immediate neighbors. However, the convergence is rather slow, i.e., $\frac{\log k}{\sqrt{k}}$, due to diminishing stepsizes. 
{ A random walk incremental strategy was used in \cite{maowalkman} to design a communication-efficient asynchronous decentralized optimization strategy, where the algorithm admits a constant stepsize to achieve exact convergence. The work in \cite{chenlag} considered a communication scenario where a central server does not periodically request gradients from all workers in decomposable convex optimization and established convergence results. The authors in \cite{lancommunication} co-designed the primal-dual decentralized optimization algorithm in outer loop and the inexactly subproblem-solving process in inner loop to save communication resources.}


In another line of research, event-triggered control emerges as a communication-efficient approach for large-scale network control systems { \cite{astromcomparison,tabuadaevent,wangevent}}. The idea is to generate network transmission only when the information conveyed by the message is deemed innovative to the system, and whether or not it is essential is determined via an event-triggered function that takes the deviation between the actual system state and the state just broadcast as an argument. 
The hope of event-triggered control is to reduce the communication load while largely preserving the control performance. 

Thanks to this attractive feature, event-triggered communication has been recently incorporated into decentralized optimization algorithms \cite{kajiyamadistributedsubgradient,hayashievent-triggered,lievent-triggered,liudistributedevent,liucommunication-censored,chenevent-triggered}. For example, the authors developed their event-triggered variants based on the decentralized optimization algorithm in \cite{nedicconstrainedconsensus}. Although reductions in communication were observed in numerical experiments, the convergence rates are rather slow: $\frac{\log k}{\sqrt{k}}$ in \cite{lievent-triggered}  and $\frac{1}{\log k}$ in \cite{kajiyamadistributedsubgradient}, where $k$ is the time counter, mainly due to the use of decaying stepsizes. To speed up convergence, constant stepsizes were used in event-triggered DGD \cite{liudistributedevent}. 
However, similar to standard DGD, the algorithm does not ensure exact minimization but only yields an accumulation point in a neighborhood of the global minimizer. Based on \cite{quharnessing}, the authors in \cite{hayashievent-triggered} solved this problem for strongly convex and smooth objective functions at the expense of maintaining an extra variable that tracks the global gradient using an event-triggered dynamic average consensus scheme. Recent work in \cite{liucommunication-censored} considered smooth and convex functions and presented an event-triggered decentralized ADMM that only requires each agent to route the decision variable to its neighbors and guarantees exact convergence. Convergence rates are analyzed for special strongly convex and smooth objectives. 
Furthermore, it is remarked in \cite{liucommunication-censored} that the event-triggered zero-gradient-sum decentralized optimization method in \cite{chenevent-triggered} can be seen as an event-triggered version of dual decomposition that is empirically slower than ADMM.
In these schemes, each agent at every generic time instant is required to exactly solve a subproblem, which may be not practical in most cases. 
\emph{Considering this, two questions naturally arise: 1) For general convex functions, is it possible to devise an event-triggered decentralized optimization algorithm that enjoys a competitive convergence rate even in the presence of node errors due to event-triggered communication? 2) If the objective functions exhibit some desired  properties, e.g., smooth or/and strongly convex, is it possible to simplify the subproblem-solving process to simple algebraic operations without sacrificing the convergence rate?}

We give affirmative answers to these questions in this work. First, the primal-dual methodology introduced earlier is used to tackle the decentralized optimization problem. More specifically, the linearized augmented Lagrangian method (LALM) in the recent work \cite{xuaccelerated} with a specific pre-conditioning strategy is used to design a periodic decentralized algorithm. 
Then, 
each agent employs an event-triggered broadcasting strategy to communicate with its neighbors
to avoid unnecessary network utilization.
{ 
Compared to the state-of-the-art, the developed event-triggered method features the following. 1) It ensures exact minimization with individual constant stepsizes to improve the speed that is usually dictated by the slowest agent in existing methods. This is made possible by adjusting the diagonal entries in the weight matrix to approximate the curvature of the objective.
2) It provides tailored local iteration rules for composite and purely smooth problems to ease computational burden. 3) Convergence rates for different types of objective functions have been established for the first time,
that is,
a convergence rate of $O(\frac{1}{k})$ for nonsmooth objective functions and a linear rate for strongly convex and smooth ones. 
To achieve this, a significantly different analysis from LALM is carried out since triggering schedulers inject errors into each iteration.
In particular, we established a new upper bound for the effect of errors on the primal-dual residual. Based on this, the same convergence rate $O(\frac{1}{k})$ with the standard primal-dual algorithm can be guaranteed for nonsmooth convex problems. }
{ It is worth to mention that recently the authors in \cite{licola} developed an event-triggered decentralized optimization method based on linearized ADMM, where subproblems are also not needed at each iteration. However, the results therein did not consider nonsmooth objectives and only presented convergence rate for strongly convex and smooth objectives. }


The notations adopted in this work are explained as follows.
We use $\mathbf{1}$ to denote a column vector with all entries being $1$, where the dimension shall be understood from the context.
Given $k_1,k_2\in\mathbb{N}$, the notation $\mathbb{N}_{[k_1,k_2]}$ represents the set $\{k\in\mathbb{N}|k_1\leq k\leq k_2\}$.
For a set $\mathcal{A}$, let $\lvert \mathcal{A} \rvert$ denote its cardinality.
Given a symmetric matrix $P$ and a column vector $x$, $P\succ0$ ($P\succeq0$) means that the matrix is positive (semi)definite; $\overline{\lambda}(P)$ and $\underline{\lambda}(P)$ stand for $P$'s maximum and minimum eigenvalues, respectively; $\big\lVert x\big\rVert$ represents the Euclidean norm;  $\big\lVert P\big\rVert$ is the corresponding induced norm; $\big\lVert x\big\rVert_P=\sqrt{x^\mathrm{T}Px}$ denotes the $P$-weighted norm. 
For a proper, lower-semicontinuous convex function $g: \mathbb{R}^m\rightarrow\mathbb{R}\cup\{+\infty\}$, the proximal operator associated with $g$  is defined by: $\mathrm{prox}_g(y)=\arg\min_x\{ f(x)+\frac{1}{2}\lVert y-x\rVert^2 \}$.
Finally, the Kronecker product of two matrices is denoted by $\otimes$.

\section{Problem Statement and Preliminaries}
\subsection{Problem statement}
This work considers the large-scale optimization problem given by
\begin{equation} \label{composite_optimization}
\min_{\theta\in\mathbb{R}^m}   \sum_{i=1}^{n}F_i(\theta)
\end{equation}
where $F_i: \mathbb{R}^m\rightarrow\mathbb{R}\cup\{+\infty\}$, $i\in\mathbb{N}_{[1,n]}$ represents the 
local objective function and $\theta$ the common decision variable.
{ 
In such problems, the number of local functions is usually large.}
Assume there are a group of agents that are connected via a network to solve \eqref{composite_optimization}, each of which, say $i$, has only access to $F_i$. 
In particular, the communication network is characterized by a simple undirected graph $\mathcal{G}=(\mathcal{V},\mathcal{E})$. Each node $i\in \mathcal{V}$ and edge $(i,j) \in \mathcal{E}$ in $\mathcal{G}$ stand for each agent $i$ and the communication channel between agents $i$ and $j$, respectively. 
Moreover, agent $j$ is said to be a neighbor of $i$ if $(i,j)\in\mathcal{E}$. We let $\mathcal{N}_i\subset\mathcal{V}$ denote the set of neighbors of $i$. 
The standard definitions of three $n\times n$ matrices are recalled:
The adjacency matrix $\mathcal{A}=[a_{ij}]$ where each entry $a_{ij}=1$ if $(i,j)\in \mathcal{E}$ and $a_{ij}=0$ otherwise, the diagonal degree matrix $\mathcal{D}={\rm diag}(\lvert\mathcal{N}_1 \rvert\cdots \lvert\mathcal{N}_n \rvert)$, and the graph Laplacian $\mathcal{L}=\mathcal{D}-\mathcal{A}$. 
Note that, for undirected graphs, the matrix $\mathcal{L}$ is ensured to be positive semidefinite. We make the following assumption for the communication graph.
\begin{assumption}\label{graphconnected}
	$\mathcal{G}=(\mathcal{V},\mathcal{E})$ is fixed and connected.
\end{assumption}

In this work, each local objective further assumes the following structure:
\begin{equation*}
F_i(\theta ) = f_i(\theta)+g_i(\theta), i\in\mathbb{N}_{[1,n]}
\end{equation*}
where $f_i: \mathbb{R}^m\rightarrow\mathbb{R}$ is smooth and convex while $g_i: \mathbb{R}^m\rightarrow\mathbb{R}\cup\{+\infty\}$ is convex but possibly non-differentiable. This problem, known as composite optimization { \cite{nesterovsmooth}}, finds wide applications in signal processing and cooperative control of multi-agent systems, e.g., the data fitting problem with $f_i$ being the loss function and $g_i$ the regularizer. By letting $f_i(\theta)=0$ or $g_i(\theta)=0, i\in\mathbb{N}_{[1,n]}$, the composite setting reduces to the special nonsmooth and smooth optimization.

{ 
	
	
	Our goal is to design an event-triggered decentralized first-order algorithm with individual constant stepsizes for the convex optimization problem in \eqref{composite_optimization} to save communication resources. In this framework, tailored local implementations are expected for composite and purely smooth problems to ease computational load. Furthermore, we will rigorously analyze the effect of triggering behavior on the iterates and therefore establish convergence rates for different types of objective functions.
	
}

\subsection{Primal-dual formulation of decentralized optimization}
Define $x=[x_1^{\mathrm{T}},\cdots,x_n^{\mathrm{T}}]^{\mathrm{T}}$,
$
F(x)=\sum_{i=1}^{n}F_i(x_i)$,
$f(x)=\sum_{i=1}^{n}f_i(x_i)$,
and $g(x)=\sum_{i=1}^{n}g_i(x_i)
$.
{ Following \cite{seidmanachebyshev},
the problem in \eqref{composite_optimization} can be equivalently written as the following linear equality-constrained optimization problem 
\begin{equation}\label{equality_constrained}
\begin{split}
\min_{{x}\in\mathbb{R}^{mn}} F({x})\quad\mathrm{s.t.} \quad\big( \sqrt{\mathcal{L}}\otimes I_m\big){x}=0.
\end{split}
\end{equation}
The augmented Lagrangian for \eqref{equality_constrained} is written as
\begin{equation*}
 F({x})+\Big\langle{y}, \big( \sqrt{\mathcal{L}}\otimes I_m\big){x}\Big\rangle+\frac{\beta}{2} \big\lVert{x}  \big\rVert^2_{ \mathcal{L}\otimes I_m}
\end{equation*}
where $y=[y_1^{\mathrm{T}},\cdots,y_n^{\mathrm{T}}]^{\mathrm{T}}\in \mathbb{R}^{mn}$ denotes the dual variable and $\beta>0$ a designable parameter.}
The KKT conditions can be identified as
\begin{subequations}\label{KKT}
	\begin{align}
	0&\in \partial F(x^*)+\big( \sqrt{\mathcal{L}}\otimes I_m\big){y}^*\label{KKT1}\\
	0&=\big( \sqrt{\mathcal{L}}\otimes I_m\big){x}^*  
	\end{align}
\end{subequations}
where $(x^*,y^*)$ is an optimal primal-dual pair and $\partial F(x^*)$ the set of all subgradients of $F$ evaluated at $x^*$.

In the remaining sections, we will focus on the case with $m=1$ to ease notation, i.e., $\mathbf{1}\otimes I_m=\mathbf{1}$, $\mathcal{L}\otimes I_m=\mathcal{L}$. The developed results can be extended to the case $m>1$ without much effort.


%
%

\section{Algorithm Development}
In this section, we develop a new periodic decentralized optimization algorithm and an event-triggered form of it to achieve a better tradeoff between network utilization and convergence speed. Some discussions about how this work relates to some recent works are presented.

\subsection{Development of an event-triggered decentralized optimization algorithm}
{ Based on the above primal-dual formulation, we recruit the LALM \cite{xuaccelerated} to solve the decentralized composite optimization problem in \eqref{composite_optimization}:
\begin{equation}\label{conditioned_LALM}
\begin{split}
{x}_{k+1} = & \arg \min_{{x}} \bigg\{ \Big\langle \triangledown f({{x}}_k)+\sqrt{\mathcal{L}}{y}_k,{x} \Big\rangle+g(x)  \\
&+\frac{\beta}{2} \big\lVert {x} \big\rVert_{\mathcal{L}}^2+\frac{1}{2} \big\lVert {x}-{x}_k  \big\rVert_{H-\beta\mathcal{L}}^2\bigg\}\\	
{y}_{k+1}=&{y}_k+\beta\sqrt{\mathcal{L}}{x}_{k+1},
\end{split}
\end{equation}
where the weight matrix used for the quadratic approximation of $f$ is designated as $H-\beta\mathcal{L}$ with $H=\mathrm{diag}(\eta_1,\cdots,\eta_n)\succ0$ being a diagonal matrix. Further by letting
$
{z}=\sqrt{\mathcal{L}}{y}
$ \cite{seidmanachebyshev,latafatnew,uribeadual},
the iteration rule becomes
\begin{equation*}
\begin{split}
{x}_{k+1} = & \arg \min_{{x}} \bigg\{ g(x)+\frac{1}{2}\big\lVert { x}-x_k\lVert_H^2 \\
&+\Big\langle \triangledown f({{x}}_k)+{z}_k+ {\beta} \mathcal{L}{x}_k,{x} \Big\rangle  \bigg\}\\
{z}_{k+1}=&{z}_k+\beta{\mathcal{L}}{x}_{k+1}.	
\end{split}
\end{equation*}}
Element-wisely,
\begin{equation}\label{original_iteration}
\begin{split}
{x}_{i,k+1} = & \arg \min_{{x}} \bigg\{ g_i(x)+\frac{\eta_i}{2}\lVert x-x_{i,k}\rVert^2 \\
&+\Big\langle \triangledown f_i({{x}}_{i,k})+{z}_{i,k}+\beta\sum_{j\in\mathcal{N}_i}\big({x}_{i,k}-{x}_{j,k}\big),{x} \Big\rangle	\bigg\} \\
		z_{i,k+1}=&z_{i,k}+\beta \sum_{j\in\mathcal{N}_i}\big({x}_{i,k+1}-{x}_{j,k+1}\big).
\end{split}
\end{equation}

It is straightforward to check that \eqref{original_iteration} can be executed in a fully decentralized fashion.
However, each agent has to broadcast its local estimate at every generic time instant $k$. Next, we exploit the fact that possibly at some time instants the progress locally made by each agent is not sufficiently significant to be sent out to save communication resources.

{ 
Denote the set of generic time instants by $\kappa =\{k|k\in\mathbb{N}\}$. It serves a global clock that synchronizes all agents. At each time $k$, we define the true and broadcast primal variable $x_{i,k}$ and $\tilde{x}_{i,k}$ for agent $i$. It is worth to mention that the variable $\tilde{x}_{i,k}$ is the same across all $j\in\mathcal{N}_i$.
Let $\kappa_i=\{k_i^l| l\in \mathbb{N}\}\subseteq \kappa$ be the set of triggering time instants for agent $i$, each of which obeys
\begin{equation} \label{triggeringrule}
{k}^{l+1}_{i} =\mathop{\min} \big \{k\in\kappa|k>{k}^{l}_{i},\big\lVert x_{i,{k}}-\tilde{x}_{i,{k-1}} \big\rVert >E_{i,k} \big \},
\end{equation}
where $k\geq 1$ and $E_{i,k}\geq0$ represents the triggering threshold.
The broadcast variable $\tilde{x}_{i,k}$ 
is defined as 
\begin{equation*} 
\tilde{x}_{i,{k}}=\begin{cases}
x_{i,{k}}, \quad \quad \; \; \, k\in\kappa_{i} \\
\tilde{x}_{i,{k-1}}, \quad {\rm otherwise} .
\end{cases}
\end{equation*}
If all the agents trigger network transmissions at time $k=0$, then $\kappa_i$ and $\tilde{x}_{i,k}$ are well-defined.
%
At each time $k$, each agent $i$ maintains the following variables:
\begin{enumerate}
	\item $i$'s local primal variable: $x_{i,k}$;
	\item $i$'s local dual variable: $z_{i,k}$;
 	\item  local primal variable broadcast by $i$: $\tilde{x}_{i,k}$;
 	\item local primal variable broadcast by $j$: $\tilde{x}_{j,k}, j\in\mathcal{N}_i$.
\end{enumerate}}

As briefly explained earlier, the degree of innovation of $x_{i,k}$ to the overall system is measured by the deviation between it and the estimate broadcast most recently, i.e., $\tilde{x}_{i,k}$. If the gap is large enough, $x_{i,k}$ will be deemed as novel and sent out. 
It can be verified from the definition that the deviation between $x_{i,k}$ and $\tilde{x}_{i,k}$ is always bounded from above by $E_{i,k}$, that is,
\begin{equation*}
\big\lVert x_{i,{k}}-\tilde{x}_{i,{k}} \big\rVert \leq E_{i,k}.
\end{equation*}
For the triggering threshold, we make the following assumption.
\begin{assumption}\label{trigthresum}
	Define $E_{k}=\max_{i\in\mathbb{N}_{[1,n]}}E_{i,k}$ for all $k\in\mathbb{N}$. $E_{k}$ is non-increasing and summable, i.e.,
	$ \sum_{k=0}^{\infty}E_{k}<\infty$.
\end{assumption}

\newtheorem{remark}{Remark}
\begin{remark}
	{ Assumption \ref{trigthresum} assumes that the threshold sequence should converge sufficiently fast. This is necessarily made to ensure exact convergence of the algorithm.}
	Examples of such sequences include $\{\frac{E_0}{k^p}\}$ where $p>1$ and $\{E_0\rho^k\}$ where $\rho<1$. 
		The non-increasing property of $E_k$ assumed in Assumption \ref{trigthresum} can be relaxed to that $E_k$ is upper bounded by a non-increasing and summable sequence without affecting the theoretical results given later. { Essentially, Assumption \ref{trigthresum} requires the bound of individual triggering thresholds to be asymtotically decreasing, implying that the event-triggered scheme asymptotically converges to its periodic counterpart. Throughout this work, we assumed that there is a global clock that synchronizes all agents. This can be satisfied by an initialization step without losing the decentralized structure.
			Based on it, each agent can construct its own triggering scheduler $E_{i,k}$ that satisfies Assumption \ref{trigthresum}. Then $E_k$ becomes the maximum of them, and Assumption \ref{trigthresum} holds true.}
\end{remark}

Based on the above communication pattern, an event-triggered decentralized optimization algorithm is formulated in Algorithm \ref{event_triggered_algorithm}.

\begin{algorithm}
	\begin{algorithmic}[1]
		\caption{ Event-triggered decentralized optimization algorithm}
		\label{event_triggered_algorithm}
		\STATE  Set $k=0$, { $z_{i,0}=0, \forall i\in\mathcal{V}$};
		each agent $i\in\mathcal{V}$ broadcasts $x_{i,0}$ to its neighbors

		\FOR{each agent $i\in\mathcal{V}$}  
		\STATE  Update primal variable
		\begin{equation}\label{event-triggered variant}
		\begin{split}
		&{x}_{i,k+1} =  \arg \min_{{x}} \bigg\{ g_i(x)+\frac{\eta_i}{2}\lVert x-x_{i,k}\rVert^2 \\
		&+\Big\langle \triangledown f_i({{x}}_{i,k})+{z}_{i,k}+\beta\sum_{j\in\mathcal{N}_i}\big(\tilde{x}_{i,k}-\tilde{x}_{j,k}\big),{x} \Big\rangle	\bigg\} 		
		\end{split}
		\end{equation}
		\STATE Test the event condition in \eqref{triggeringrule};
		\IF {triggered}
		\STATE Broadcast $x_{i,{k+1}}$ to its neighbors;
		\ENDIF
		\STATE  Update dual variable
		\begin{equation*}
		z_{i,k+1}=z_{i,k}+\beta \sum_{j\in\mathcal{N}_i}\big(\tilde{x}_{i,k+1}-\tilde{x}_{j,k+1}\big)
		\end{equation*}
		\ENDFOR 
		\STATE Set $k = k+1$.
	\end{algorithmic}
\end{algorithm}

Note that if the overall objective has Lipschitz continuous gradients, i.e., $g(x)=0$, the iteration for primal variable can be further simplified. To see this, we set $g_i(x)=0$ and can get
the following closed-form solution for \eqref{event-triggered variant}:
\begin{equation}\label{primal_update_smooth}
\begin{split}
x_{i,k+1}=    x_{i,k} -\frac{1}{\eta_i}\Big({z}_{i,k}+\triangledown f_i({{x}}_{i,k})+\beta \sum_{j\in\mathcal{N}_i}\big(\tilde{x}_{i,k}-\tilde{x}_{j,k}\big) \Big).
\end{split}
\end{equation}
This helps lower down the computational load required to exactly solve a minimization subproblem when the objective is smooth. In addition, the algorithm also works for completely nonsmooth functions, i.e., $f(x)=0$, with the following modifications in Step 3:
\begin{equation*}
\begin{split}
{x}_{i,k+1} = & \arg \min_{{x}} \bigg\{ g_i(x)+\frac{\eta_i}{2}\lVert x-x_{i,k}\rVert^2 \\
&+\Big\langle {z}_{i,k}+\beta\sum_{j\in\mathcal{N}_i}\big(\tilde{x}_{i,k}-\tilde{x}_{j,k}\big),{x} \Big\rangle	\bigg\}. 
\end{split}
\end{equation*}
{ 
\begin{remark}
	Note that the iteration rule in \eqref{event-triggered variant} can be rewritten as
	\begin{equation*}
	\begin{split}
	&x_{i, k+1} =\\ 
	&   \mathrm{prox}_{\frac{1}{\eta_i}g_i}\Big( x_{i,k}-\frac{z_{i,k}}{\eta_i}-\frac{\triangledown f_i(x_{i,k})}{\eta_i} -\frac{ \beta}{\eta_i} \sum_{j\in\mathcal{N}_i}(\tilde{x}_{i,k}-\tilde{x}_{j,k}) \Big),
	\end{split}
	\end{equation*}
	which helps make the implementation much easier in some common optimization problems, e.g., $l_1$-regularized problems. 
	From this perspective, the periodic algorithm in \eqref{original_iteration} can be seen as a proximal gradient descent step with stepsize $\frac{1}{\eta_i}$ applied on the augmented Lagrangian followed by a dual gradient ascent step with stepsize $\beta$. The method in \cite{shiapro} has the same primal-dual structure, as explained in \cite{liaprimal}.
\end{remark}}

{ 
\begin{remark}
	In each iteration, each agent employs the
	stale iterate $\tilde{x}_{i,k}$ instead of the updated one ${x}_{i,k}$ to construct $\sum_{j\in\mathcal{N}_i}\big(\tilde{x}_{i,k}-\tilde{x}_{j,k}\big)$. This is mainly motivated by the fact that
	in purely event-triggered consensus 
	the use of $\tilde{x}_{i,k}$ helps conserve the average of variables:
	\begin{equation*}
		\mathbf{1}^{\mathrm{T}}x_{k+1} = \mathbf{1}^{\mathrm{T}}\big(x_{k}-\beta\mathcal{L}\tilde{x}_k\big) =\mathbf{1}^{\mathrm{T}}x_k.
	\end{equation*}
If we use $\sum_{j\in\mathcal{N}_i}\big({x}_{i,k}-\tilde{x}_{j,k}\big)$ in the iteration, then according to the optimality condition we have
\begin{equation*}
\begin{split}
0 =& \triangledown f({x}_k)+z_k+\tilde{\triangledown} g(x_{k+1}) +H (x_{k+1}-x_k) \\
&+\beta \mathcal{L}\tilde{x}_k-\beta \mathcal{D}e_k\\
0 =& z_{k+1}-z_k-{\beta}\mathcal{L}\tilde{x}_{k+1}+\beta \mathcal{D}e_{k+1}
\end{split}
\end{equation*}
and therefore
\begin{equation*}
\begin{split}
0  =&\triangledown f({x}_k)+\tilde{\triangledown} g(x_{k+1})+z_{k+1}  +P\big( x_{k+1}-x_k \big)\\
&+\beta \mathcal{A}\big( e_{k+1}-e_k\big).
\end{split}
\end{equation*}	
In light of this, we can conclude that the main results developed in this work remain valid with a slightly different definition for ``$a$" used to describe the optimization error in Theorem \ref{Convergence_smooth}.

\end{remark}}


\subsection{Connection with the event-triggered ADMM \cite{liucommunication-censored}}
The proposed method generalizes the recent work in \cite{liucommunication-censored}. To see this, we recall the decentralized ADMM in \cite{shionthelinear}
\begin{equation*}
\begin{split}
x_{k+1}=& (2c\mathcal{D}+\partial F)^{-1}\big( c(\mathcal{D}+\mathcal{A})x_k-z_k \big) \\
z_{k+1}=& z_k+c(\mathcal{D}-\mathcal{A})x_{k+1}
\end{split}
\end{equation*}
and its event-triggered version in \cite{liucommunication-censored}
\begin{equation}\label{communication-censored ADMM}
\begin{split}
x_{k+1}=& (2c\mathcal{D}+\partial F)^{-1}\big( c(\mathcal{D}+\mathcal{A})\tilde{x}_k-z_k \big) \\
z_{k+1}=& z_k+c(\mathcal{D}-\mathcal{A})\tilde{x}_{k+1}
\end{split}
\end{equation}
where $c>0$ is the stepsize. Since $\mathcal{L}=\mathcal{D}-\mathcal{A}$, we equivalently express \eqref{communication-censored ADMM} as

\begin{equation*}
\begin{split}
x_{k+1}=& (2c\mathcal{D}+\partial F)^{-1}\big( c(2\mathcal{D}-\mathcal{L})\tilde{x}_k-z_k \big) \\
z_{k+1}=& z_k+c\mathcal{L}\tilde{x}_{k+1}
\end{split}
\end{equation*}
and therefore
\begin{equation}\label{proximal_form_ADMM}
\begin{split}
x_{k+1}=& (cI_n+\frac{\mathcal{D}^{-1}}{2}\partial F)^{-1}\big( c\tilde{x}_k-\frac{c}{2}\mathcal{D}^{-1}\mathcal{L}\tilde{x}_k- \frac{\mathcal{D}^{-1}}{2}z_k \big) \\
\frac{\mathcal{D}^{-1}}{2}z_{k+1}=& \frac{\mathcal{D}^{-1}}{2}z_k+\frac{c}{2}\mathcal{D}^{-1}\mathcal{L}\tilde{x}_{k+1}.
\end{split}
\end{equation}
For the proposed algorithm,  according to the optimality condition \eqref{event-triggered variant}
\begin{equation*}
	0  \in \triangledown f({x}_k)+z_k+\partial  g(x_{k+1}) -H x_k +\beta \mathcal{L}\tilde{x}_k+H x_{k+1},
\end{equation*}
we have
\begin{equation}\label{proximal_form}
x_{k+1} = (H+\partial g)^{-1}\big(H x_k-\triangledown f({x}_k)-z_k-\beta \mathcal{L}\tilde{x}_k\big).
\end{equation}
By comparing \eqref{proximal_form} and \eqref{proximal_form_ADMM}, we can see that the event-triggered ADMM in \cite{liucommunication-censored} is very similar to the proposed method when the smooth part of the objective $f(x)=0$, and the only difference is that the Laplacian matrix is normalized to $\frac{c}{2}\mathcal{D}^{-1}\mathcal{L}$, the dual variable is scaled to $\frac{\mathcal{D}^{-1}}{2}z$, and the proximal parameter is changed to $(cI_n+\frac{\mathcal{D}^{-1}}{2}\partial F)^{-1}$. However, this work further establishes convergence rate for nonsmooth convex functions while \cite{liucommunication-censored} only proved convergence, and provides tailored implementations for composite and smooth objectives to reduce computational load.

\section{Main Results}
\subsection{Convergence rate for composite objective functions}

\begin{assumption}\label{Lipschitz}
	$f_i$ is a convex and Lipschitz differentiable function with positive parameter $l_{f_i}$, i.e.,
	\begin{equation*}
	\big\lVert\triangledown f_i(x)- \triangledown f_i(y) \big\rVert\leq l_{f_i} \big\lVert x-y  \big\rVert, \forall x,y\in\mathbb{R}.
	\end{equation*}
\end{assumption}
By Assumption \ref{Lipschitz} and the Cauchy-Schwartz inequality, we have that the gradient of $f(x)$ satisfies
\begin{equation*}
\langle \triangledown f(x)- \triangledown f(y)
,x-y\rangle\leq  \lVert x-y \rVert^2_{L_f}, \forall x,y\in\mathbb{R}^{n},
\end{equation*}
where $L_f=\mathrm{diag}(l_{f_1},\cdots,l_{f_n})$.
Then, we examine the convergence rate for Algorithm \ref{event_triggered_algorithm} with Assumptions \ref{graphconnected}-\ref{Lipschitz} satisfied.

\newtheorem{theorem}{$Theorem$}
\begin{theorem}\label{Convergence_smooth}
	If Assumptions \ref{graphconnected}-\ref{Lipschitz} hold and
\begin{equation*}
P-{L_f}\succ0,
\end{equation*}
then 
	\begin{equation}\label{error_rate}
	\begin{split}
	\big	\lVert \sqrt{\mathcal{L}}{ \hat{x}}_{t} \big\rVert \leq   \frac{\big({\big\lVert x_{0}-x^* \big\rVert_{P} 
			+\rho\big\lVert {\mathcal{L}(\beta \mathcal{L}+\frac{\mathbf{1}\mathbf{1}^{\mathrm{T}}}{n})^{-1}} \big\rVert}+\sqrt{2b}A_t \big)^2}{2t\big(\rho-\big\lVert  y^*\big\rVert\big)},
	\end{split}
	\end{equation} 
	and 
	\begin{equation}\label{objective_rate}
	\begin{split}
	&- \frac{\big\lVert y^* \big\rVert\big({\big\lVert x_{0}-x^* \big\rVert_{P} 
			+\rho\big\lVert {\mathcal{L}(\beta \mathcal{L}+\frac{\mathbf{1}\mathbf{1}^{\mathrm{T}}}{n})^{-1}} \big\rVert}+\sqrt{2b}A_t \big)^2}{2t\big(\rho-\big\lVert  y^*\big\rVert\big)}\\ &\leq  F({ \hat{x}}_t)-F(x^*) \\
	&	\leq  \frac{\big({\big\lVert x_{0}-x^* \big\rVert_{P} 
			+\rho\big\lVert {\mathcal{L}(\beta \mathcal{L}+\frac{\mathbf{1}\mathbf{1}^{\mathrm{T}}}{n})^{-1}} \big\rVert}+\sqrt{2b}A_t \big)^2}{2t},
	\end{split}
	\end{equation}
	where 	$
	P = H-\beta\mathcal{L}
	$, ${ \hat{x}}_t=\frac{1}{t}\sum_{k=1}^{t}x_{k}$, { $\big\lVert y^* \big\rVert<\rho <\infty$}, 
 $
	A_t=\frac{2a\sqrt{n}}{b}\sum_{k=1}^{t}E_{k-1}
$, $a=\max\big\{2\beta\overline{\lambda}(\mathcal{L}),1\big\}$, and $b=\min\big\{\underline{\lambda}(L_f),\frac{1}{\overline{\lambda}(\beta \mathcal{L}+\frac{\mathbf{1}\mathbf{1}^{\mathrm{T}}}{n})}\big\}$.
\end{theorem}
\begin{proof}
Please refer to Appendix A.
\end{proof}

The results in Theorem \ref{Convergence_smooth} are explained in what follows.

\emph {1) Comparison of sufficient conditions with \cite{liucommunication-censored}:}
Theorem \ref{Convergence_smooth} states that both the consensus error $\big\lVert  \sqrt{\mathcal{L}}{ \hat{x}}_t\big\rVert$ and the objective error $F({ \hat{x}}_t)-F(x^*)$ converge to zero at an ergodic convergence rate of $O(\frac{1}{t})$ if some reasonable assumptions hold true. 
It is worth to mention that the result remains valid for smooth objective functions, i.e.,  $g(x)=0$, with a much simplified iteration rule in \eqref{primal_update_smooth}. For completely nonsmooth objective functions, i.e., $f(x)=0$, the condition for stepsize to ensure the same convergence rate is relaxed to 
\begin{equation*}
H-\beta\mathcal{L}\succ 0.
\end{equation*}
Note that the diagonal matrix $H$ allows the use of different stepsizes for agents, depending on the local Lipschitz modulus.
By setting $H=\eta I_n$, the sufficient condition reduces to
\begin{equation*}
\eta > \beta \overline{\lambda}(\mathcal{L}).
\end{equation*}
When the free parameter $\beta$ approaches zero,
this condition becomes equivalent to that in \cite{liucommunication-censored} for convergence.

{ 
\emph {2) Impact of free parameter $\rho$:}
		Theoretically speaking, the $\rho$ in Theorem 1 can be any constant that is larger than the $2$-norm of optimal dual variable $y^*$. If $\rho\rightarrow \infty$, then one has
		\begin{equation*}
		\begin{split}
		&\lim\limits_{\rho\rightarrow\infty}\big	\lVert \sqrt{\mathcal{L}}{ \hat{x}}_{t} \big\rVert \\
		\leq& \lim\limits_{\rho\rightarrow\infty}   \frac{\big({\big\lVert x_{0}-x^* \big\rVert_{P} 
				+\rho\big\lVert {\mathcal{L}(\beta \mathcal{L}+\frac{\mathbf{1}\mathbf{1}^{\mathrm{T}}}{n})^{-1}} \big\rVert}+\sqrt{2b}A_t \big)^2}{2t\big(\rho-\big\lVert  y^*\big\rVert\big)}\\
		=&  \lim\limits_{\rho\rightarrow\infty} \frac{\big( { \frac{\big\lVert x_{0}-x^* \big\rVert_{P} }{\rho}
				+\big\lVert {\mathcal{L}(\beta \mathcal{L}+\frac{\mathbf{1}\mathbf{1}^{\mathrm{T}}}{n})^{-1}} \big\rVert}+\frac{\sqrt{2b}A_t}{\rho} \big)^2}{2t\big(\frac{1}{\rho}-\frac{\big\lVert  y^*\big\rVert}{\rho^2}\big)}\\
		=&\infty.
		\end{split}
		\end{equation*} 
		This implies that generally a larger $\rho$ leads to a looser upper bound for $\lVert \sqrt{\mathcal{L}}{\hat{x}}_{t} \big\rVert$ and when $\rho$ approaches $\infty$ the bound becomes meaningless. However, as long as $\rho$ is finite, the bound is in the order of $O(\frac{1}{t})$. By following the same line of reasoning, similar results hold for the bound on $\lvert F({\hat{x}}_t)-F(x^*)\rvert$.}

{ 
\emph {3) Choices of $\beta$, $H$ and $E_k$:}
	Theorem \ref{Convergence_smooth} reveals that, given a graph Laplacian, a larger $\beta$ leads to an $H$ with larger diagonal entries that is used in the quadratic approximation. If $H$ over-approximates the curvature of $f$ in \eqref{conditioned_LALM}, the convergence of primal variables $x_k$ will be slow. However, if $\eta_i$ is too small and under-approximates the curvature, the primal iterate may oscillate quickly.
	In practice, the designable matrix $H$ and parameter $\beta$ should be carefully tuned to achieve a reasonable convergence rate. 
An appropriate choice would be to set $\beta=\frac{1}{\overline{\lambda}(\mathcal{L})+1}$ and $H= I_n+L_f$. Theorem \ref{Convergence_smooth} also suggests that the threshold sequence $E_k$ will affect convergence. In particular, a more slowly decreasing $E_k$ satisfying Assumption \ref{trigthresum} will result in a larger $A_t$ and therefore a larger base in convergence constants. For composite optimization, one can set a threshold sequence that converges slightly faster than the guaranteed rate $O(\frac{1}{t})$, e.g., $\frac{E_{i,0}}{t^2}$. In addition, a base constant $E_{i,0}$ that is sufficiently smaller than the magnitude of ${z}_{i,0}+\triangledown f_i({{x}}_{i,0})+\beta \sum_{j\in\mathcal{N}_i}\big({x}_{i,0}-{x}_{j,0}\big)$ is suggested to prevent wild oscillation in the beginning.}

\newtheorem{lemma}{$Lemma$}

\subsection{Linear convergence for strongly convex and smooth objective functions}
This subsection considers strongly convex and smooth objective functions, for which stronger convergence results can be expected. Formally, the following assumption is made for the objective functions.
\begin{assumption}\label{strongly convex}
	For all $i\in\mathbb{N}_{[1,n]}$, $g_i(\theta)=0$ and $f_i$ is strongly convex with positive parameter $\mu_{f_i}$, i.e.,
	\begin{equation*}
	\langle \triangledown f_i(x)- \triangledown f_i(y),x-y \rangle 
\geq \mu_{f_i} \big\lVert x-y  \big\rVert^2, \forall x,y\in\mathbb{R}.
	\end{equation*}
\end{assumption}
As a direct consequence, the gradient of $f(x)$ satisfies
	\begin{equation*}
\langle \triangledown f(x)- \triangledown f(y),x-y \rangle 
\geq\big\lVert x-y  \big\rVert_M^2, \forall x,y\in\mathbb{R}^n,
\end{equation*}
where $M=\mathrm{diag}(\mu_{f_1},\cdots,\mu_{f_n})$.

\begin{theorem}\label{linear_convergence}
	If Assumptions \ref{graphconnected}-\ref{strongly convex} hold and
\begin{equation}\label{stepsize_stronglyconvex}
P-\frac{L_f^2}{k_1}\succ0
\end{equation}
	for some $0<k_1<2\underline{\lambda}(M)$, then there exists some positive $\sigma$ such that
	\begin{equation}\label{linear convergence}
	\begin{split}
	&\frac{1}{2}\big\lVert x_{k}-x^* \big\rVert_{P +k_4Q}^2+\frac{1}{2}\big\lVert z_{k}-z^*\big\rVert_{\big(\beta \mathcal{L}+\frac{\mathbf{1}\mathbf{1}^{\mathrm{T}}}{n}\big)^{-1}}^2 +nC E_k^2 \\ &\geq 
	\frac{\sigma+1}{2} \Big(\big\lVert x_{k+1}-x^* \big\rVert_{P +k_4Q}^2    + \big\lVert z^*-z_{k+1}\big\rVert_{\big(\beta \mathcal{L}+\frac{\mathbf{1}\mathbf{1}^{\mathrm{T}}}{n}\big)^{-1}}^2\Big)
	\end{split}
	\end{equation}
	where 
	$
	P = H-\beta\mathcal{L}, Q = 2M-k_1I_n
	$,$0<k_2$, $2<k_3$, $0< k_4< 1$, $0<k_5$,
	and
	\begin{equation*}
	\begin{split}
	C =&\bigg(\frac{2\big({k_3}+\frac{1}{k_2} -1 \big)(\sigma+k_5) \overline{\lambda}(\beta^2\mathcal{L}^2)}{(	1-\frac{2}{k_3})\underline{\lambda}(\beta \mathcal{L}+\frac{\mathbf{1}\mathbf{1}^{\mathrm{T}}}{n}) }    + \frac{\overline{\lambda}(\beta \mathcal{L}+\frac{\mathbf{1}\mathbf{1}^{\mathrm{T}}}{n})}{2k_5} \\
	&+ \frac{2\overline{\lambda}(\beta^2\mathcal{L}^2)}{k_5\underline{\lambda}(P +k_4Q)} \bigg).
	\end{split}
	\end{equation*} 
\end{theorem}
\begin{proof}
	The proof is postponed to Appendix B.
\end{proof}

{ 
	\begin{remark}
		The analysis framework developed in this work may apply to some asynchronous cases, e.g., \cite{jakovetic,chenlag}. 
		In this work, the triggering scheduler imposes conditions on the outdated information deemed effective, that is, the error between it and the real-time information should be decreasing fast enough (summable). A rigorous analysis that heavily exploits this property is then carried out. In particular, the effect of triggering behavior on the primal-dual residual is proved bounded when the triggering threshold is summable over time. For some asynchronous communication patterns such as \cite{jakovetic,chenlag}, the algorithms can be expressed as inexact executions of standard ones. For example, the authors in \cite{jakovetic} considered a scenario where each working agent may be inactive with a decreasing probability over time, and treated the missed information as errors. Similarly, the authors in \cite{chenlag} used outdated information to save resources provided that the progress made during two sequential instants is bounded from above. 
		However, for some asynchronous algorithms, this framework doe not apply. For example, the design and analysis of the asynchronous algorithm in \cite{maowalkman} are motivated by incremental methods, and therefore are essentially different from this work.
\end{remark}}

Some discussions on the results in Theorem \ref{linear_convergence} are in order.

\emph {1) Linear convergence:}
	It is revealed in Theorem \ref{linear_convergence} that if the objective function is further assumed to be strongly convex and the stepsize satisfies a relatively stricter condition then a much faster convergence rate can be obtained. In particular, if $E_k$ linearly converges then we obtain a linear convergence rate for the primal-dual residual. And if the rate constant for a linearly convergent $E_k$ is smaller than $\sqrt{\frac{1}{1+\sigma}}$ then the convergence of the primal-dual residual is linear with constant $\frac{1}{1+\sigma}$ as in a periodic algorithm.

{ 
\emph {2) Impact of free parameters:}
	In Theorem \ref{linear_convergence}, several free parameters such as $k_1,k_2,k_3,k_4,k_5$ are used to describe convergence results. 
	How the specific values of them affect the result is explained in the following.
\begin{itemize}
	\item 	$k_1$ is used in \eqref{k1} and should be selected in set $(0, 2\underline{\lambda}(M))$. It directly affects the choices of $H$ and $\beta$, as suggested by the sufficient condition in \eqref{stepsize_stronglyconvex} for convergence.
	
	\item $k_2$ and $k_3$ are used in \eqref{intermediate_result} to separate triggering errors from the primal-dual residual, and should be chosen in $(0,\infty)$ and $(2,\infty)$, respectively. A smaller $k_2$ and a larger $k_3$ give a conservative constant $C$ in \eqref{linear convergence}, but allow us to get a larger $\sigma$ and therefore faster convergence.
	
	\item $k_4$ should be in $(0,1)$. Its role can be observed in \eqref{key} and \eqref{key_inequality}, where the relation between $\big\lVert P\big( x_{k+1}-x_{k} \big) \big\rVert^2$ and the primal-dual residual is established. A larger $k_4$ renders the weight on the primal residual in \eqref{linear convergence} heavier, but makes $\sigma$ smaller to get \eqref{key} satisfied.
	
	\item The proof shows that the key for linear convergence is the satisfaction of \eqref{key}
	with a sufficiently small $\sigma+k_5$.
	This implies that $\sigma+k_5$ can only take values in
	\begin{equation}\label{k_range}
	(0,\min\{R_1,R_2,R_3\}),
	\end{equation} 
	where
	\begin{equation*}
	\begin{split}
	R_1&\triangleq \frac{ \underline{\lambda}\big(k_4Q{(L_f^{-1})}^2\big) (1-\frac{2}{k_3})\underline{\lambda}(\beta\mathcal{L}+\frac{\mathbf{1}\mathbf{1}^{\mathrm{T}}}{n})}{k_2+k_3-1}\\
	R_2&\triangleq \underline{\lambda}\Big(P^{-1}-\frac{L_f(P^{-1})^2}{k_1}\Big) \underline{\lambda}(\beta\mathcal{L}+\frac{\mathbf{1}\mathbf{1}^{\mathrm{T}}}{n})(1-\frac{2}{k_3})\\
	R_3& \triangleq (1-k_4)Q(P+k_4Q)^{-1},
	\end{split}
	\end{equation*}
	to ensure \eqref{key}.
	Therefore setting a larger $k_5$ in \eqref{k_range} leads to a smaller $\sigma$ and slower convergence. In particular, given $k_5$ and a linearly decreasing threshold $\{E_0\rho^k\}$, the convergence rate becomes
	\begin{equation*}
	O\Big(\max\Big\{  \rho^2, \frac{1}{\min\{R_1,R_2,R_3\}-k_5} \Big\}^k\Big).
	\end{equation*}
\end{itemize}}


{ 
\emph {3) Choices of $\beta$, $H$ and $E_k$:}
 Selecting a larger $\eta_i$ and a smaller $\beta$ generally leads to heavier weights, i.e., $P+k_4Q$ and $\big(\beta\mathcal{L}+\frac{\mathbf{1}\mathbf{1}^{\mathrm{T}}}{n}\big)^{-1}$,  on the primal and dual residuals  in \eqref{linear convergence}. However, a larger spectral radius of $H-\beta\mathcal{L}$ also results in a smaller $\min(R_1,R_2,R_3)$ in \eqref{k_range} and therefore slower convergence. For the reasonable choices of these parameters, one can set $\beta=\frac{1}{\overline{\lambda}(\mathcal{L})+1}$ and $H= I_n+\frac{L_f^2}{k_1}$. 
 Since the slower one in $E_k$ and ${\frac{1}{(1+\sigma)^{0.5k}}}$ will dominate the convergence,
it is then always preferable to choose an exponentially decreasing  sequence for the triggering threshold in distributed strongly convex optimization. For the base constant, $E_{i,0}=0.1\lvert{z}_{i,0}+\triangledown f_i({{x}}_{i,0})+\beta \sum_{j\in\mathcal{N}_i}\big({x}_{i,0}-{x}_{j,0}\big)\lvert$ is an appropriate choice.
	
}

\section{Simulation Studies}

This section examines the effectiveness of the proposed algorithm by applying it to minimize two types of objective functions, that is, composite convex objectives and strongly convex and smooth objectives. The simulations were performed using Matlab R2017b on macOS Catalina with an Intel I9 processor of 2.3 GHz and 16 GB RAM.

{ 
\subsection{Case \uppercase\expandafter{\romannumeral1}: Composite convex objectives }
Consider the decentralized $l_1$-$l_2$ minimization problem: 
\begin{equation*}
\min_{\theta} \sum_{i=1}^{n}\bigg\{
\frac{1}{2}\lVert b_i - A_i \theta\rVert^2 + \tau_i \lVert \theta\rVert_1\bigg\},
\end{equation*}
where data $A_i\in\mathbb{R}^{p_i\times m}$, $b_i\in\mathbb{R}^{p_i}$ and regularization parameter $\tau_i>0$ are private to agent $i$. The two component functions for each agent are $f_i(\theta)=\frac{1}{2}\lVert b_i - A_i \theta\rVert^2$ that is convex with Lipschitz continuous gradient, and $g_i(\theta)=\tau_i\lVert \theta \rVert_1$ that is convex but nondifferentiable. In the simulation, the parameters are chosen as $p_i=3$, $m=50$, and $n=100$; the data $A_i$ and $b_i$ are randomly generated with normalization. 

In the simulation, a network of $n=100$ agents is randomly chosen with connectivity ratio $r=0.4$ \cite{wattscollective}, where $r$ is defined as the number of links divided by the number of all possible links $\frac{n(n-1)}{2}$. 
We compare the performance of the proposed methods with the ADMM-based algorithm \cite{shionthelinear} and  its event-triggered variant (COCA) \cite{liucommunication-censored}. For \cite{shionthelinear,liucommunication-censored}, the projected scaled subgradient method available as a Matlab function $L1General2\_PSSgb$ in \cite{schmidtfast} is used to solve the subproblems with an accuracy of $10^{-10}$ in terms of the $l_\infty$ norm of the subgradient. 
Communication strategies in which each agent triggers network transmission every two iterations or four iterations are also simulated. 
The parameters for these algorithms are manually tuned  in periodic setting to achieve the best performance: $H=0.6I_n, \beta=0.0025$ and $c=0.0025$ are considered for the proposed method and \cite{shionthelinear,liucommunication-censored}, respectively.
For event-triggered methods,  the triggering thresholds for agents are set as $E_{i,k}=\frac{20}{k^{1.2}}$.
The initial guesses of primal and dual variables are set as $0$ for all methods. The performances are evaluated in terms of the objective error $\lvert F(\hat{x}_t)-F(x^*) \rvert$ over the number of iteration steps and broadcasting times of the first agent.

The results are plotted in Figs. \ref{iteration0.4} and \ref{communication0.4}. We can see from them that both event-triggered LALM and COCA, while having comparable performances with their periodic counterparts, achieve significant communication reductions. 
In this example, the event-triggered LALM has slightly better performance than periodic LALM, implying that the effect of triggering
behavior on the performance is non-monotone.
The results also suggest that ADMM-based approaches outperform LALM-based ones in terms of convergence rate. 
This is mainly because that  ADMM-based methods used the original augmented Lagrangian while in the proposed method a linearized one is used to ease the computational burden of solving subproblems.
As a consequence, ADMM and COCA consume much more computational resources than the proposed methods at each iteration.
In this specific example, the time spent per iteration for COCA is $0.2431\rm{s}$ on average and the time for the proposed method is $0.0068\rm{s}$. In practice, a tradeoff between network utilization and computational resource consumption should be made. The periodic scheme of $2$ periods halves the number of communication rounds for each agent. However, when the number of periods increases to $4$, the iterates diverge. Compared to the periodic scheme of $2$ periods, the proposed algorithm consumes less communication cost and has convergence rate guarantees. 

Then, a sparser random network with $n=100$ and $r=0.04$ is considered. The parameters are tuned as $H=0.6I_n, \beta=0.01$ to achieve the best performance. 
The results in Figs. \ref{iteration0.04} and \ref{communication0.04} suggest that 
the denser configuration ($r=0.4$) leads to faster convergence, and each agent broadcasts more in sparser networks to achieve a given accuracy. This is primarily because that a denser network has a more balanced set of weights for agents and more information from neighbors can be used in each iteration/communication round. }

\begin{figure}[!htb]
	\centering%
	\includegraphics[width=3.5in]{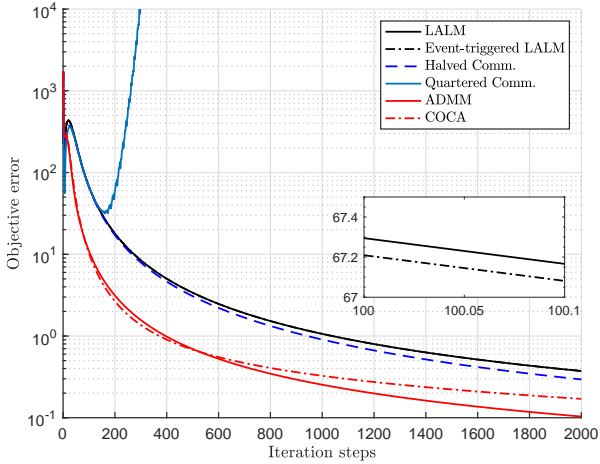}
	\caption{Objective error $\lvert F(\hat{x}_t)-F(x^*) \rvert$ versus iteration steps when $r=0.4$.}\label{iteration0.4}
\end{figure}

\begin{figure}[!htb]
	\centering%
	\includegraphics[width=3.5in]{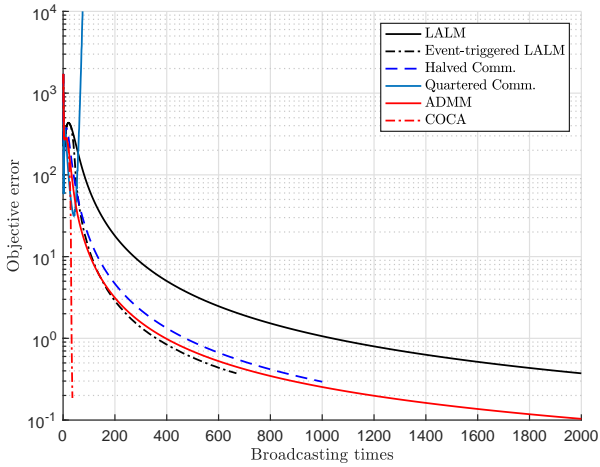}
	\caption{Objective error $\lvert F(\hat{x}_t)-F(x^*) \rvert$ versus broadcasting times when $r=0.4$.}\label{communication0.4}
\end{figure}

\begin{figure}[!htb]
	\centering%
	\includegraphics[width=3.5in]{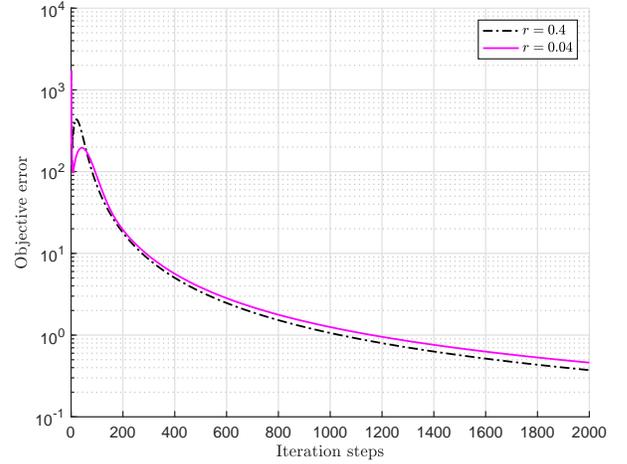}
	\caption{Objective error $\lvert F(\hat{x}_t)-F(x^*) \rvert$ versus iteration steps in different random networks.}\label{iteration0.04}
\end{figure}

\begin{figure}[!htb]
	\centering%
	\includegraphics[width=3.5in]{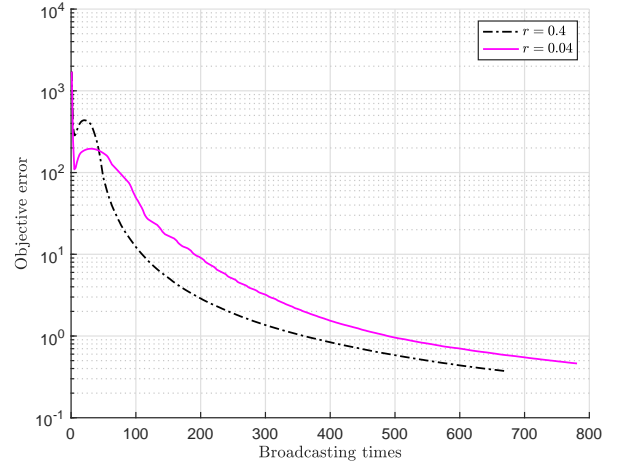}
	\caption{Objective error $\lvert F(\hat{x}_t)-F(x^*) \rvert$ versus broadcasting times in different random networks.}\label{communication0.04}
\end{figure}

	\begin{table}[!htb]
	\begin{center}
		\caption{The time spent per iteration of two algorithms}
		{\tabcolsep6pt\begin{tabular}{@{}ll@{}}\toprule
				Algorithm & Time spent per iteration\\
			\midrule
				COCA  &$0.2431\rm{s}$\\
			\midrule
			Event-triggered LALM  &  $0.0068\rm{s}$ \\ 
					\bottomrule
				\label{tab1}
		\end{tabular}}{}	
		\vspace{-2em}	
	\end{center}
\end{table}

\subsection{Case \uppercase\expandafter{\romannumeral2}: Strongly convex and smooth objectives }

Consider the following decentralized logistic regression problem:
\begin{equation*}
\min_{\theta} \sum_{i=1}^{n}\bigg\{ \sum_{j=1}^{m_i}\ln \Big(1+\exp \big(-y^i_j(M^{i\mathrm{T}}_j\theta)\big)\Big) \bigg\}.
\end{equation*}
where the input features $M_j^i\in\mathbb{R}^{m}$ and the class labels $y_j^i\in\{-1,1\}$ with $j=1,\cdots,m_i$ are privately held by each agent $i$. Note that we set the last element of the feature vector $M_j^i\in\mathbb{R}^{m}$ to $1$ as in standard logistic regression,
then the last element of the decision variable $\theta$ becomes the adjustable bias of the logistic regression model. 
The number of samples for each agent $i$ is $m_i=8$, and the dimension for decision variable is $m=10$.
In the simulation, all the $400$ samples are generated randomly. A network of $n=100$ with $r=0.04$ is considered.  

{ 
The linearized ADMM-based algorithm (DLM) in \cite{lingdlm} and the gradient tracking method in \cite{quharnessing}, and their event-triggered variants \cite{licola} (COLA) and  \cite{hayashievent-triggered} are simulated for comparison.
Their parameters are manually tuned in periodic setting to achieve the best performance: $H=55I_n, \beta=1$ for the proposed method, $c=1, \rho=50$ for \cite{lingdlm,licola}, and
$\eta=0.06$ for \cite{quharnessing,hayashievent-triggered}. The mixing matrix in \cite{quharnessing,hayashievent-triggered} is selected with the Metropolis rule \cite{xiaofast}.
The local initial guesses of primal and dual variables for each agent are set as $0$.
The triggering threshold for exchanging primal variables is set as $E_k={0.9^{0.1k}}$ for all event-triggered methods.
An extra triggering threshold used to track the gradient in \cite{hayashievent-triggered} is selected as ${0.3^{0.1k}}$.
We evaluate the performance by considering the residual $\frac{\lVert x_k-x^*\rVert_F}{\lVert x_0-x^*\rVert_F}$ over the number of local iteration iteration steps and communication times of the first agent. 


\begin{figure}[!htb]
	\centering%
	\includegraphics[width=3.5in]{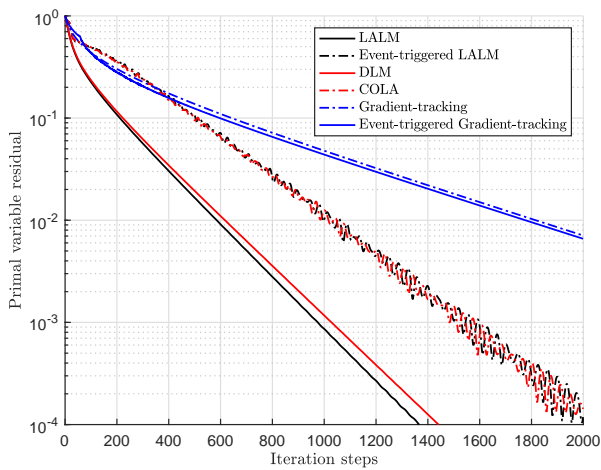}
	\caption{Primal variable residual versus iteration steps when $r=0.04$.}\label{iteration_strongly}
\end{figure}

\begin{figure}[!htb]
	\centering%
	\includegraphics[width=3.5in]{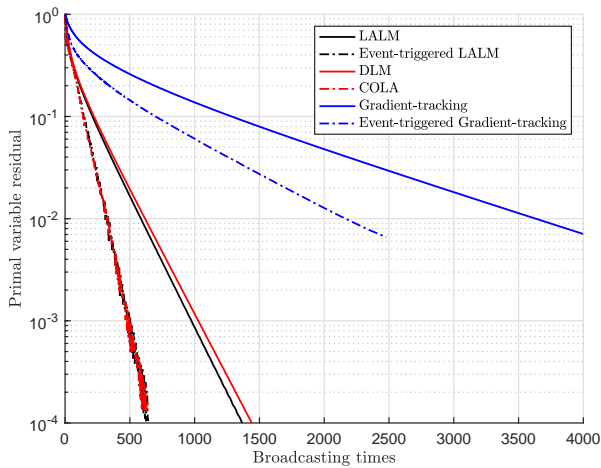}
	\caption{Primal variable residual versus broadcasting times when $r=0.04$.}\label{communication_strongly}
\end{figure}


The results are reported in Figs. \ref{iteration_strongly} and \ref{communication_strongly}. 
Both the methods exactly converge.
However, the gradient-tracking method converges at a much slower rate than other two types of methods. This is primarily because that this algorithm only allows one parameter to be tuned while other methods have two.
The results also show that generally event-triggered methods converge at slower rates and present more oscillatory trajectories than their periodic counterparts, mainly due to the errors caused by event-triggering communication.
However, significant reductions in network utilization are observed in event-triggered methods.
In particular, the proposed method and COLA
save $\sim\frac{1}{2}$ communication cost to achieve an accuracy of $10^{-4}$. The gradient-tracking method consumes much heavier communication cost since both the estimated gradient and decision variable have to be exchanged.}

\section{Conclusion}
In this work, we have designed a decentralized event-triggered algorithm for large-scale convex optimization problems with coupled cost functions. Convergence rates of the proposed algorithm have been established for different types of objectives with particular triggering thresholds. Numerical experiments have demonstrated the effectiveness of the proposed method in saving communication cost. 
{ 
In the literature of event-triggered control, triggering schedulers and system states are generally related. 
Relating schedulers to the decision variable at the last
triggering instance may help achieve a better communication-rate tradeoff. Establishing a communication complexity bound for reaching a given accuracy in terms of the triggering threshold is also important. These topics will be explored in the future.}


%

\appendices
\section{}
Before developing the proof for Theorem \ref{Convergence_smooth}, several useful technical lemmas are presented. 

\begin{lemma} \cite{xuaccelerated}\label{symmetric} 
	Given a positive semidefinite matrix $W\in\mathbb{R}^{n\times n}$, it holds
		\begin{equation}\label{useful_equality}
	2\Big\langle Wu,v \Big\rangle =\big\lVert u \big\rVert^2_W+\big\lVert v \big\rVert^2_W-\big\lVert u-v \big\rVert^2_W, \forall u,v\in\mathbb{R}^n.
	\end{equation}
\end{lemma}
\begin{lemma}\cite{xuabregman}\label{uniqueness_span1}
	If Assumption \ref{graphconnected} holds, then for each $y\in  span^{\perp}{\mathbf{1}}$, there exists a unique $y'\in span^{\perp}{\mathbf{1}}$ such that $y=\mathcal{L}y'$ and vice versa.
\end{lemma}

\begin{lemma}\label{one_iteration}
	If all the conditions in Theorem \ref{Convergence_smooth} hold, then, for any $x\in null(\mathcal{L})$ and $z\in span^{\perp}{\mathbf{1}}$,
	\begin{equation}
	\begin{split}
	&F({x}_{k+1})-F(x)+{  \Big\langle z,  {x}_{k+1} \Big\rangle}\\
	\leq& -\frac{1}{2}\big\lVert x_{k+1}-x_k \big\rVert_{P-{L_f}}^2 -\frac{1}{2}\big( \big\lVert x_{k+1}-x \big\rVert_{P}^2-\big\lVert x_{k}-x \big\rVert_{P}^2 \big) \\
	&-\Big\langle x_{k+1} -x, \beta \mathcal{L}\big(e_k-e_{k+1}\big)\Big\rangle+\Big\langle e_{k+1} , z_{k+1}-z\Big\rangle \\
	& + \frac{1}{2}\big(\big\lVert z-z_{k}\big\rVert_{(\beta \mathcal{L}+\frac{\mathbf{1}\mathbf{1}^{\mathrm{T}}}{n})^{-1}}^2 - \big\lVert z-z_{k+1}\big\rVert_{(\beta \mathcal{L}+\frac{\mathbf{1}\mathbf{1}^{\mathrm{T}}}{n})^{-1}}^2 \big)\\
	&-\frac{1}{2}\big\lVert z_{k+1}-z_{k}\big\rVert^2_{(\beta \mathcal{L}+\frac{\mathbf{1}\mathbf{1}^{\mathrm{T}}}{n})^{-1}}
	\end{split}
	\end{equation}
	where $P$ is defined in Theorem \ref{Convergence_smooth}  and $e_k=\tilde{x}_k-x_k$.
\end{lemma}
\begin{proof}[Proof of Lemma \ref{one_iteration}]
	By the smoothness of $f$, we have
	\begin{equation*}
	\begin{split}
	&f({x}_{k+1})\\
	\leq& f({x}_k)+ \Big\langle \triangledown f({x}_k),{x}_{k+1}-{x}_k  \Big\rangle +\frac{1}{2} \big\lVert {x}_{k+1}-{x}_k  \big\rVert _{L_f}^2.
	\end{split}
	\end{equation*}
	It follows from the convexity of $f$ 
	\begin{equation*}
	f(x_k)+\Big\langle \triangledown f(x_k),x-x_k \Big\rangle\leq f(x)
	\end{equation*}
	and $g$
	\begin{equation*}
 g(x_{k+1})-g(x) \leq 	\Big\langle x_{k+1}-x,\tilde{\triangledown}g(x_{k+1})	\Big\rangle 
	\end{equation*} 
	that
	\begin{equation}\label{saddle_point}
	\begin{split}
&F({x}_{k+1})-F(x)\\
	\leq & \Big\langle \triangledown f(x_k)+\tilde{\triangledown}g(x_{k+1}), x_{k+1}-x  \Big\rangle  +\frac{1}{2} \big\lVert {x}_{k+1}-{x}_k  \big\rVert _{L_f}^2.
	\end{split}
	\end{equation}
	From the iteration rule, we have
	\begin{equation*}
	\begin{split}
	0 =& \triangledown f({x}_k)+z_k+\tilde{\triangledown} g(x_{k+1}) -Hx_k +\beta \mathcal{L}\tilde{x}_k+H x_{k+1} \\
	0 =& z_{k+1}-z_k-{\beta}\mathcal{L}\tilde{x}_{k+1}
	\end{split}
	\end{equation*}
	where $\tilde{\triangledown} g(x_{k+1})$ is a subgradient of $g$ evaluated at $x_{k+1}$. 
	This implies 
	\begin{equation}\label{optimality_condition}
	\begin{split}
	0 =& \triangledown f({x}_k)+\tilde{\triangledown} g(x_{k+1})+z_{k+1}  +P\big( x_{k+1}-x_k \big)\\
	&+\beta \mathcal{L}\big( e_k-e_{k+1}\big).
	\end{split}
	\end{equation}
	Calculating the inner products of $x_{k+1}-x$ with both sides 
	give rise to
	\begin{equation}\label{optimality_inner_product}
	\begin{split}
	& \Big\langle x_{k+1} -x,\triangledown f({x}_k)+\tilde{\triangledown}g(x_{k+1})\Big\rangle \\
	=&- \Big\langle x_{k+1}-x , z_{k+1}-z\Big\rangle -\Big\langle x_{k+1} -x,P\big(x_{k+1}-x_k\big)\Big\rangle \\
	&-\Big\langle x_{k+1} -x, \beta \mathcal{L}\big(e_k-e_{k+1}\big)\Big\rangle -\Big\langle x_{k+1} , z\Big\rangle
	\end{split}
	\end{equation}
	for any $x\in null(\mathcal{L})$ and $z\in span^{\perp}{\mathbf{1}}$.
	From Lemma \ref{uniqueness_span1} and { the fact that
	\begin{equation*}
	z_{t+1}=\beta\mathcal{L}\sum_{k=0}^{t+1}\tilde{x}_{k},
	\end{equation*} }we obtain
	\begin{equation}
	\begin{split}\label{further_transformation}
	&\Big\langle x_{k+1}-x , z_{k+1}-z\Big\rangle \\
	&=\Big\langle \tilde{x}_{k+1} , z_{k+1}-z\Big\rangle-\Big\langle e_{k+1} , z_{k+1}-z\Big\rangle\\
	& = \Big\langle \beta\mathcal{L}\tilde{x}_{k+1} , z_{k+1}'-z'\Big\rangle-\Big\langle e_{k+1} , z_{k+1}-z\Big\rangle \\
	& =  \Big\langle \beta \mathcal{L}\big(z'_{k+1}-z'_k\big) , z_{k+1}'-z'\Big\rangle-\Big\langle e_{k+1} . z_{k+1}-z\Big\rangle.
	\end{split}
	\end{equation}
It follows
	\begin{equation}
	\begin{split}\label{intermediate_result}
	&F({x}_{k+1})-F(x)+ \Big\langle z,  {x}_{k+1}\Big \rangle \\
\overset{\romannumeral1}{	\leq}& \frac{ 1}{2}\big\lVert x_{k+1}-x_k \big\rVert_{L_f}^2-\Big\langle x_{k+1} -x,P\big(x_{k+1}-x_k\big)\Big\rangle \\
	&-\Big\langle x_{k+1} -x, \beta \mathcal{L}\big(e_k-e_{k+1}\big)\Big\rangle+\Big\langle e_{k+1} , z_{k+1}-z\Big\rangle\\
	&- \Big\langle \beta\mathcal{L}\big(z'_{k+1}-z'_k\big) , z_{k+1}'-z'\Big\rangle \\
\overset{\romannumeral2}{	=}&  -\frac{ 1}{2}\big\lVert x_{k+1}-x_k \big\rVert^2_{P-{L_f}} -\frac{1}{2}\big( \big\lVert x_{k+1}-x \big\rVert_{P}^2-\big\lVert x_k-x \big\rVert_{P}^2 \big) \\
	&-\Big\langle x_{k+1} -x, \beta \mathcal{L}\big(e_k-e_{k+1}\big)\Big\rangle+\Big\langle e_{k+1} , z_{k+1}-z\Big\rangle\\
	& + \frac{1}{2}\big( \big\lVert z'-z'_{k}\big\rVert_{{\beta \mathcal{L}}}^2 - \big\lVert z'_{k+1}-z'\big\rVert_{{\beta \mathcal{L}}}^2 \big)-\frac{1}{2}\big\lVert z'_{k+1}-z'_{k}\big\rVert^2_{\beta \mathcal{L}},
	\end{split}
	\end{equation}
	where we plug \eqref{optimality_inner_product} and \eqref{further_transformation} into \eqref{saddle_point} to get ``${\romannumeral1}$" and use Lemma \ref{symmetric} and 
	\begin{equation*}
	P\succ {L_f}\succeq 0,\, \beta \mathcal{L}\succeq 0 
	\end{equation*}
to get ``${\romannumeral2}$".
	{ Due to $z,z',z_k'\in span^{\perp}{\mathbf{1}}$, we have
	\begin{equation*}
	z-z_k=\beta \mathcal{L}\big(z'-z_k'\big)=\big(\beta \mathcal{L}+\frac{\mathbf{1}\mathbf{1}^{\mathrm{T}}}{n}\big)\big(z'-z_k'\big).
	\end{equation*}
	and therefore $z'-z_k'=\big(\beta \mathcal{L}+\frac{\mathbf{1}\mathbf{1}^{\mathrm{T}}}{n}\big)^{-1}\big(z-z_k\big)$. Then we consider
\begin{equation*}
	\big\lVert z'-z'_k\big\rVert_{{\beta \mathcal{L}}}^2 = \Big\langle z-z_k, z'-z_k'\Big\rangle= \big\lVert z-z_k\big\rVert_{\big(\beta \mathcal{L}+\frac{\mathbf{1}\mathbf{1}^{\mathrm{T}}}{n}\big)^{-1}}^2,
\end{equation*}}
	which together with \eqref{intermediate_result} gives the desired inequality.
\end{proof}

\begin{lemma}\label{primal-dual_bound}
	If all the conditions in Theorem \ref{Convergence_smooth} hold, then, for $k\leq t$, 
	\begin{equation*}
	\begin{split}
	&\big\lVert x_{k}-x^* \big\rVert+\big\lVert z^*-z_{k}\big\rVert \\
	&\leq  
	2A_t+  
	\sqrt{\frac{{2}}{b}}\big(\big\lVert x_{0}-x^* \big\rVert_{P}+\big\lVert z_0-z^* \big\rVert_{(\beta \mathcal{L}+\frac{\mathbf{1}\mathbf{1}^{\mathrm{T}}}{n})^{-1}}\big)
	\end{split}
	\end{equation*}
	where $P$, $b$ and $A_t$ are defined in Theorem \ref{Convergence_smooth}.
\end{lemma}
\begin{proof}[Proof of Lemma \ref{primal-dual_bound}]
{ 	First, we use the convexity of $F$ and the KKT conditions \eqref{KKT} to obtain
	\begin{equation}\label{primal_dual_error}
	\begin{split}
	&F(x)-F(x^*)+\Big\langle y^*,\big(\sqrt{\mathcal{L}}\otimes I_m\big)x\Big\rangle\\
	\geq &\Big \langle \tilde{\triangledown} F(x^*), x-x^*\Big\rangle+\Big\langle y^*,\big(\sqrt{\mathcal{L}}\otimes I_m\big)\big(x-x^*\big)\Big\rangle \\
	=& \Big \langle \tilde{\triangledown} F(x^*)+\big( \sqrt{\mathcal{L}}\otimes I_m\big){y}^*, x-x^*\Big\rangle = 0, \forall x
	\end{split}
	\end{equation}
	where $\tilde{\triangledown} F(x^*)\in \partial F(x^*)$.}
Then, we let $x=x^*$ and $z=z^*=\mathcal{L}y^*$ by definition in \eqref{one_iteration} and sum it over $k$ from $0$ to $t-1$ to get
	\begin{equation}\label{sum_primal-dual}
	\begin{split}
	& 0\leq \sum_{k=0}^{t-1}\Big(F(x_{k+1})-F(x^*)+\Big\langle y^*,\sqrt{\mathcal{L}}x_{k+1}\Big\rangle\Big)\\
	\leq&
	-\frac{ 1}{2}\Big(\sum_{k=0}^{t-1}\big\lVert x_{k+1}-x_k \big\rVert^2_{P-{L_f}} +\big\lVert x_{t}-x^* \big\rVert_{P}^2 -\big\lVert x_{0}-x^* \big\rVert_{P}^2\Big)  \\
	&-\frac{1}{2}\big\lVert z^*-z_{t}\big\rVert_{(\beta \mathcal{L}+\frac{\mathbf{1}\mathbf{1}^{\mathrm{T}}}{n})^{-1}}^2+\frac{1}{2}\big\lVert z^*-z_0 \big\rVert^2_{(\beta \mathcal{L}+\frac{\mathbf{1}\mathbf{1}^{\mathrm{T}}}{n})^{-1}}\\
	&+\sum_{k=0}^{t-1}\Big( \Big\langle x^*-x_{k+1},\beta\mathcal{L}(e_k-e_{k+1}) \rangle +\langle z_{k+1}-z^*,e_{k+1} \Big\rangle\Big)\\
	&-\frac{1}{2}\sum_{k=0}^{t-1}\big\lVert z_{k+1}-z_{k}\big\rVert^2_{(\beta \mathcal{L}+\frac{\mathbf{1}\mathbf{1}^{\mathrm{T}}}{n})^{-1}}.
	\end{split}
	\end{equation}
	Since $P-{L_f}\succ 0$ and $(\beta \mathcal{L}+\frac{\mathbf{1}\mathbf{1}^{\mathrm{T}}}{n})^{-1}\succ 0$, it holds
	\begin{equation*}
	\begin{split}
	&\frac{1}{2}\big\lVert x_{t}-x^* \big\rVert_{P}^2+\frac{1}{2}\big\lVert z^*-z_{t}\big\rVert_{(\beta \mathcal{L}+\frac{\mathbf{1}\mathbf{1}^{\mathrm{T}}}{n})^{-1}}^2 \\
	\leq&  \frac{1}{2}\big\lVert x_{0}-x^* \big\rVert_{P}^2+\frac{1}{2}\big\lVert z_0-z^* \big\rVert^2_{(\beta \mathcal{L}+\frac{\mathbf{1}\mathbf{1}^{\mathrm{T}}}{n})^{-1}}\\
	&
	+\sum_{k=1}^{t}\Big( \Big\langle x^*-x_{k},\beta \mathcal{L}\big(e_{k-1}-e_{k}\big) \Big\rangle +\Big\langle z_{k}-z^*,e_{k} \Big\rangle\Big) .
	\end{split}
	\end{equation*}
	By the monotonicity of $ E_k $ and the Cauchy-Schwarz inequality, we further have
	\begin{equation}\label{intermediate_bound_primal-dual}
	\begin{split}
	&\frac{1}{4}\min\big\{\underline{\lambda}(L_f),\frac{1}{\overline{\lambda}(\beta \mathcal{L}+\frac{\mathbf{1}\mathbf{1}^{\mathrm{T}}}{n})}\big\}\big(  \big\lVert x_{t}-x^* \big\rVert+\big\lVert z^*-z_{t}\big\rVert \big)^2 \\
	\leq& \frac{1}{2}\big\lVert x_{t}-x^* \big\rVert_{P}^2+\frac{1}{2}\big\lVert z^*-z_{t}\big\rVert_{(\beta \mathcal{L}+\frac{\mathbf{1}\mathbf{1}^{\mathrm{T}}}{n})^{-1}}^2 \\
	&
	+\sum_{k=1}^{t}\Big( \Big\langle x^*-x_{k},\beta \mathcal{L}\big(e_{k-1}-e_{k}\big) \Big\rangle +\Big\langle z_{k}-z^*,e_{k} \Big\rangle\Big) \\
	\leq&  \frac{1}{2}\big\lVert x_{0}-x^* \big\rVert_{P}^2+\frac{1}{2}\big\lVert z_0-z^* \big\rVert^2_{(\beta \mathcal{L}+\frac{\mathbf{1}\mathbf{1}^{\mathrm{T}}}{n})^{-1}}\\
	&
	+\sum_{k=1}^{t} \max\big\{2\beta\overline{\lambda}(\mathcal{L}),1\big\}\sqrt{n}E_{k-1}\big(\big\lVert x^*-x_{k} \big\rVert+\big\lVert z_{k}-z^*\big\rVert\big) .
	\end{split}
	\end{equation}
	Upon using Lemma 1 in \cite{schmidtconvergence}, we obtain
	\begin{equation*}
	\begin{split}
	&\big\lVert x_{t}-x^* \big\rVert+\big\lVert z^*-z_{t}\big\rVert\\
	\leq & A_t+ \Big(  
	2\frac{\big\lVert x_{0}-x^* \big\rVert_{P}^2+\big\lVert z_0-z^* \big\rVert^2_{(\beta \mathcal{L}+\frac{\mathbf{1}\mathbf{1}^{\mathrm{T}}}{n})^{-1}}}{b}+A_t^2\Big)^{1/2}
	\end{split}
	\end{equation*}
	where $b$ and $A_t$ are defined in Theorem \ref{Convergence_smooth}.
	By the monotonicity and positivity of $A_t$, the desired result follows.
\end{proof}

We are now in a position to present the proof for Theorem \ref{Convergence_smooth}.

\begin{proof}[Proof of Theorem \ref{Convergence_smooth}]
	Manipulating \eqref{sum_primal-dual} and using the similar procedure as in \eqref{intermediate_bound_primal-dual} allow us to get
	\begin{equation*}
	\begin{split}
	&\frac{ 1}{2}\sum_{k=0}^{t-1}\big(\big\lVert x_{k+1}-x_k \big\rVert^2_{P-{L_f}}+\big\lVert z_{k}-z_{k+1}\big\rVert^2_{(\beta \mathcal{L}+\frac{\mathbf{1}\mathbf{1}^{\mathrm{T}}}{n})^{-1}}\big)\\
	\leq&\frac{1}{2}\big\lVert x_{0}-x^* \big\rVert_{P}^2 +\frac{1}{2}\big\lVert z_0-z^* \big\rVert^2_{(\beta \mathcal{L}+\frac{\mathbf{1}\mathbf{1}^{\mathrm{T}}}{n})^{-1}}\\
	&+\big(\big\lVert x^*-x_{t} \big\rVert+\big\lVert z^*-z_{t}\big\rVert\big)\sum_{k=1}^{t} a\sqrt{n}E_{k-1} 
	\end{split}
	\end{equation*}
	where $a$ is defined in Theorem \ref{Convergence_smooth}.
	In light of Lemma \ref{primal-dual_bound}, we have that if $E_{k-1}$ is summable, then 
	\begin{equation*}
	\sum_{k=0}^{\infty}\big(\big\lVert x_{k+1}-x_k \big\rVert^2_{P-{L_f}}+\big\lVert z_{k}-z_{k+1}\big\rVert^2_{(\beta \mathcal{L}+\frac{\mathbf{1}\mathbf{1}^{\mathrm{T}}}{n})^{-1}}\big)<\infty.
	\end{equation*}
	Since $P-{L_f}\succ0,\big(\beta \mathcal{L}+\frac{\mathbf{1}\mathbf{1}^{\mathrm{T}}}{n}\big)^{-1}\succ0$, we further have
	\begin{equation*}
	\lim\limits_{k\rightarrow\infty}\big(x_{k+1},z_{k+1}\big)-\big(x_{k},z_{k}\big)=0.
	\end{equation*}
	Denote the limit point of $\big\{\big(x_k,z_{k}\big)\big\}_{k\geq 1}$ by $\big(x_{\infty},z_{\infty}\big)$. Note that $\lim\limits_{k\rightarrow\infty}E_k=0$ by assumptions. From 
	\begin{equation*}
	\beta \mathcal{L}\big(e_{k+1}+x_{k+1}\big)=z_{k+1}-z_k,
	\end{equation*}
	and 
	\begin{equation*}
	0 = \tilde{\triangledown} F(x_{k+1})+z_k -H{x}_k+\beta \mathcal{L}\big(e_k+x_k\big)+H x_{k+1}
	\end{equation*}
	where $\tilde{\triangledown} F(x_{k+1})$ is a subgradient of $F$ evaluated at $x_{k+1}$,
	we obtain $\mathcal{L}x_{\infty}=0$ and $\tilde{\triangledown} F(x_\infty)+z_\infty=0$, respectively. This implies that $\big(x_\infty,y_\infty\big)$ is a KKT point, where ${z_\infty}=\sqrt{\mathcal{L}}{y_\infty}$. Again, from \eqref{sum_primal-dual}, we have
	\begin{equation}\label{convergence}
	\begin{split}
	&\sum_{k=0}^{t-1}\Big(F(x_{k+1})-F(x^*)+\Big\langle y^*,\sqrt{\mathcal{L}}x_{k+1}\Big\rangle\Big)\\
	\leq & \frac{1}{2}\big\lVert x_{0}-x^* \big\rVert_{P}^2 
	+\frac{1}{2}\big\lVert z_0-z^* \big\rVert^2_{(\beta \mathcal{L}+\frac{\mathbf{1}\mathbf{1}^{\mathrm{T}}}{n})^{-1}} 
	\\
	&	+\frac{b}{2}A_t\Big(2A_t+  
	\sqrt{\frac{{2}}{b}}\big(\big\lVert x_{0}-x^* \big\rVert_{P}+\big\lVert z_0-z^* \big\rVert_{(\beta \mathcal{L}+\frac{\mathbf{1}\mathbf{1}^{\mathrm{T}}}{n})^{-1}}\big)\Big)\\
	\leq & \Big(\frac{\big\lVert x_{0}-x^*\big \rVert_{P} 
		+\big\lVert z_0-z^* \big\rVert_{(\beta \mathcal{L}+\frac{\mathbf{1}\mathbf{1}^{\mathrm{T}}}{n})^{-1}}}{\sqrt{2}}+\sqrt{b}A_t \Big)^2,
	\end{split}
	\end{equation}
	which in conjunction with
	\begin{equation*}
	\begin{split}
	&t \Big( F({ \hat{x}}_{t})-F(x^*) + \Big\langle y^*,\sqrt{\mathcal{L}}{ \hat{x}}_{t} \Big\rangle \Big)\\
	\leq &
	\sum_{k=0}^{t-1}\Big( F(x_{k+1})-F(x^*) + \Big\langle y^*,\sqrt{\mathcal{L}}x_{k+1} \Big\rangle \Big)
	\end{split}
	\end{equation*}
	gives
	\begin{equation*}
	\begin{split}
	&F({ \hat{x}}_{t})-F(x^*) + \Big\langle y^*,\sqrt{\mathcal{L}}{ \hat{x}}_{t} \Big\rangle \\
	\leq & \frac{1}{2t}\big({\big\lVert x_{0}-x^* \big\rVert_{P} 
		+\big\lVert z_0-z^* \big\rVert_{(\beta \mathcal{L}+\frac{\mathbf{1}\mathbf{1}^{\mathrm{T}}}{n})^{-1}}}+\sqrt{2b}A_t \big)^2 \\
	= &  \frac{1}{2t}\big({\big\lVert x_{0}-x^* \big\rVert_{P} 
		+\big\lVert \sqrt{\mathcal{L}}y^* \big\rVert_{(\beta \mathcal{L}+\frac{\mathbf{1}\mathbf{1}^{\mathrm{T}}}{n})^{-1}}}+\sqrt{2b}A_t \big)^2,
	\end{split}
	\end{equation*}
	where $z_0=0$ for initialization is used to get the last equality.
	{ 
	Finally, we consider
	\begin{equation}
	\begin{split} \label{primal_dual_error_2}
	&F({\hat{x}}_{t})-F(x^*)\leq F({\hat{x}}_{t})-F(x^*) + \rho\big\lVert \sqrt{\mathcal{L}}{\hat{x}}_{t} \big\rVert \\
	\leq & \sup_{\lVert y^*\lVert \leq \rho} \frac{\big({\big\lVert x_{0}-x^* \big\rVert_{P} 
			+\big\lVert y^* \big\rVert_{\mathcal{L}(\beta \mathcal{L}+\frac{\mathbf{1}\mathbf{1}^{\mathrm{T}}}{n})^{-1}}}+\sqrt{2b}A_t \big)^2}{2t}\\
	\leq &  \frac{\big({\big\lVert x_{0}-x^* \big\rVert_{P} 
			+\rho\big\lVert {\mathcal{L}(\beta \mathcal{L}+\frac{\mathbf{1}\mathbf{1}^{\mathrm{T}}}{n})^{-1}} \big\rVert}+\sqrt{2b}A_t \big)^2}{2t}.
	\end{split}
	\end{equation}
	By \eqref{primal_dual_error}, it holds that
	\begin{equation}\label{lower_bound}
	F({\hat{x}}_{t})-F(x^*) \geq -\big\lVert y^*\big\rVert\big\lVert\sqrt{\mathcal{L}}{ \hat{x}}_{t}\big\rVert.
	\end{equation}
	By combining \eqref{lower_bound} with \eqref{primal_dual_error_2}, one gets 
	\begin{equation*}
	\begin{split}
	&(\rho-\lVert {y^*}\rVert)\lVert\sqrt{\mathcal{L}}{ \hat{x}}_{t}\big\rVert\\
	\leq & \frac{\big({\big\lVert x_{0}-x^* \big\rVert_{P} 
			+\rho\big\lVert {\mathcal{L}(\beta \mathcal{L}+\frac{\mathbf{1}\mathbf{1}^{\mathrm{T}}}{n})^{-1}} \big\rVert}+\sqrt{2b}A_t \big)^2}{2t}.
	\end{split}
	\end{equation*}
Therefore the bound for $\big\lVert \sqrt{\mathcal{L}}{\hat{x}}_{t} \big\rVert$ in \eqref{error_rate} holds. Using \eqref{lower_bound} again allows us to obtain the lower bound for $F({\hat{x}}_t)-F(x^*)$ in \eqref{objective_rate}.} This completes the proof.
\end{proof}

\section{}
\begin{proof}[Proof of Theorem \ref{linear_convergence}]
	By setting $\tilde{\triangledown}g(x_{k+1})=0$ in \eqref{optimality_condition}
	and $z^*=\sqrt{\mathcal{L}}y^*$ in \eqref{KKT1},
	we have
	\begin{equation}\label{optimality_condition_strong_convexity}
	\begin{split}
	0 =& \triangledown f({x}_k)-\triangledown f({x}^*)+z_{k+1}-z^*  + P\big( x_{k+1}-x_k \big)\\
	&+\beta \mathcal{L}\big( e_k-e_{k+1}\big).
	\end{split}
	\end{equation}
	As in the proof of Lemma \ref{one_iteration}, we consider the inner products of $x_{k+1}-x^*$ with both sides of the above equality
	\begin{equation}\label{optimality_condition_inner_product_strongly_convexity}
	\begin{split}
	& \underbrace{\Big\langle x_{k+1} -x^*,\triangledown f({x}_k)-\triangledown f({x}^*)\Big\rangle}_{\romannumeral1} + \underbrace{\Big\langle x_{k+1}-x^* , z_{k+1}-z^*\Big\rangle }_{\romannumeral2}\\
	+&\underbrace{\Big\langle x_{k+1} -x^*,P\big(x_{k+1}-x_k\big)\Big\rangle}_{\romannumeral3}\\
	+&\Big\langle x_{k+1} -x^*, \beta\mathcal{L}\big(e_k-e_{k+1}\big)\Big\rangle= 0.
	\end{split}
	\end{equation}
	For ``${\romannumeral1}$", we consider $\triangledown f({x}_k)-\triangledown f({x}^*)= \triangledown f({x}_k)-\triangledown f({x}_{k+1})+\triangledown f({x}_{k+1})-\triangledown f({x}^*)$ and get from the strong convexity and smoothness of $f$ that
	\begin{equation}\label{k1}
	\begin{split}
	{\romannumeral1}
	\geq& \big\lVert x_{k+1}-x^* \big\rVert_M^2\\
	&-\frac{1}{2k_1}\big\lVert \triangledown f({x}_k)-\triangledown f({x}_{k+1}) \big\rVert^2 - \frac{k_1}{2}\big\lVert x_{k+1}-x^* \big\rVert^2 \\
	\geq& \frac{1}{2}\big\lVert x_{k+1}-x^* \big\rVert_{Q}^2 -\frac{1}{2k_1}\big\lVert {x}_k-{x}_{k+1} \big\rVert_{{L_f^2}}^2.
	\end{split}
	\end{equation}
	With the same reasoning as in \eqref{further_transformation}, we have
	\begin{equation}\label{dual_result_strong_convexity}
	\begin{split}
	{\romannumeral2}
	=  &\frac{1}{2}  \big\lVert z^*-z_{k+1}\big\rVert_{\big(\beta\mathcal{L}+\frac{\mathbf{1}\mathbf{1}^{\mathrm{T}}}{n}\big)^{-1}}^2+\frac{1}{2}\big\lVert z_{k}-z_{k+1}\big\rVert^2_{\big(\beta\mathcal{L}+\frac{\mathbf{1}\mathbf{1}^{\mathrm{T}}}{n}\big)^{-1}} \\
	& -\frac{1}{2}\big\lVert z_{k}-z^*\big\rVert_{\big(\beta\mathcal{L}+\frac{\mathbf{1}\mathbf{1}^{\mathrm{T}}}{n}\big)^{-1}}^2-\Big\langle e_{k+1} , z_{k+1}-z^*\Big\rangle .
	\end{split}
	\end{equation}
	Using Lemma \ref{symmetric} allows us to obtain
	\begin{equation}\label{c}
	\begin{split}
	{\romannumeral3} 
	=\frac{1}{2}\big( \big\lVert x_{k+1}-x^* \big\rVert_{P}^2 + \big\lVert x_{k+1}-x_{k} \big\rVert_{P}^2 - \big\lVert x_{k}-x^* \big\rVert_{P}^2\big).
	\end{split}
	\end{equation}
	Combing Eqs. \eqref{optimality_condition_inner_product_strongly_convexity}-\eqref{c} yields
	\begin{equation}\label{rough_decreasing}
	\begin{split}
	&\frac{1}{2}\big\lVert x_{k}-x^* \big\rVert_{P}^2+\frac{1}{2}\big\lVert z_{k}-z^*\big\rVert_{\big(\beta\mathcal{L}+\frac{\mathbf{1}\mathbf{1}^{\mathrm{T}}}{n}\big)^{-1}}^2\\
	&	+\Big\langle e_{k+1} , z_{k+1}-z^*\Big\rangle-\Big\langle x_{k+1} -x^*, \beta\mathcal{L}\big(e_k-e_{k+1}\big)\Big\rangle\\\geq& 
	\frac{1}{2} \big\lVert x_{k+1}-x^* \big\rVert_{P+Q}^2 + \frac{1}{2}\big\lVert x_{k+1}-x_{k} \big\rVert_{P-\frac{{L_f^2}}{k_1}}^2 \\
	& + \frac{1}{2}  \big\lVert z^*-z_{k+1}\big\rVert_{\big(\beta\mathcal{L}+\frac{\mathbf{1}\mathbf{1}^{\mathrm{T}}}{n}\big)^{-1}}^2+\frac{1}{2}\big\lVert z_{k}-z_{k+1}\big\rVert^2_{\big(\beta\mathcal{L}+\frac{\mathbf{1}\mathbf{1}^{\mathrm{T}}}{n}\big)^{-1}} .
	\end{split}
	\end{equation}
	
{	 In order to obtain linear convergence from \eqref{rough_decreasing}, we establish a relation between $\big\lVert  x_{k+1}-x_{k}  \big\rVert^2$ and the primal-dual residual in the following.}
	By \eqref{optimality_condition_strong_convexity} and the inequality
	\begin{equation*}
	2\Big \langle u,v \Big\rangle \geq -w \big\lVert u \big\rVert^2-\frac{1}{w} \big\lVert v\big\rVert^2, \forall u,v\in\mathbb{R}^n,w>0,
	\end{equation*} 
	it holds
	\begin{equation}\label{immediate_result}
	\begin{split}
	&\big\lVert P\big( x_{k+1}-x_{k} \big) \big\rVert^2 \\
	=& \big\lVert \triangledown f({x}_k)-\triangledown f({x}^*)+z_{k+1}-z^*  +\beta\mathcal{L}\big( e_k-e_{k+1}\big) \big\rVert^2 \\
	\geq & \big(1-k_2-k_3\big) \big\lVert {x}_k-{x}^*\big\rVert_{ L_f^2}^2 
	+\big(1-\frac{2}{k_3} \big)\big\lVert z_{k+1}-z^* \big\rVert^2  \\
	&+  \big(1-\frac{1}{k_2} -{k_3} \big)\big \lVert \beta\mathcal{L}\big( e_k-e_{k+1}\big)  \big\rVert^2
	\end{split}
	\end{equation}
	for any $k_2>0$ and $k_3>2$.
	If $\sigma+k_5$ is sufficiently small such that
	\begin{subequations}\label{key}
		\begin{align}
		\frac{(k_2+k_3-1)(\sigma+k_5) L_f^2}{(1-\frac{2}{k_3})\underline{\lambda}(\beta\mathcal{L}+\frac{\mathbf{1}\mathbf{1}^{\mathrm{T}}}{n})}&\preceq k_4Q	\label{key1} \\
		\frac{(\sigma+k_5)P^2}{(1-\frac{2}{k_3})\underline{\lambda}(\beta\mathcal{L}+\frac{\mathbf{1}\mathbf{1}^{\mathrm{T}}}{n})} &\preceq  P-\frac{ L_f^2}{k_1} \label{key2} \\
		(\sigma+k_5) (P+k_4Q) &\preceq (1-k_4)Q\label{key3}
		\end{align}
	\end{subequations}
	for some $0< k_4< 1$,
	then we can get from \eqref{immediate_result} that
	\begin{equation}
	\begin{split} \label{key_inequality}
	& \frac{1}{2} \big\lVert x_{k+1}-x^* \big\rVert_{(1-k_4)Q}^2 + \frac{1}{2}\big\lVert x_{k+1}-x_{k} \big\rVert_{P-\frac{ L_f^2}{k_1}}^2 \\
	& + \frac{1}{2} \big\lVert x_{k}-x^* \big\rVert_{k_4Q}^2+\frac{1}{2}\big\lVert z_{k}-z_{k+1}\big\rVert^2_{\big(\beta\mathcal{L}+\frac{\mathbf{1}\mathbf{1}^{\mathrm{T}}}{n}\big)^{-1}} \\
	& + \frac{\big({k_3}+\frac{1}{k_2} -1 \big)(\sigma+k_5)}{2(	1-\frac{2}{k_3})\underline{\lambda}(\beta\mathcal{L}+\frac{\mathbf{1}\mathbf{1}^{\mathrm{T}}}{n}) }\big \lVert \beta\mathcal{L}\big( e_k-e_{k+1}\big)  \big\rVert^2\\
	\geq&  \frac{\sigma+k_5}{2} \Big(\big\lVert x_{k+1}-x^* \big\rVert_{P+k_4Q }^2+\big\lVert z_{k+1}-z^*\big\rVert_{\big(\beta\mathcal{L}+\frac{\mathbf{1}\mathbf{1}^{\mathrm{T}}}{n}\big)^{-1}}^2\Big).
	\end{split}
	\end{equation}
	Combining \eqref{key_inequality} and \eqref{rough_decreasing} leads to	
	%
	\begin{equation*}
	\begin{split}
	&\frac{1}{2}\big\lVert x_{k}-x^* \big\rVert_{P+k_4Q}^2+\frac{1}{2}\big\lVert z_{k}-z^*\big\rVert_{\big(\beta\mathcal{L}+\frac{\mathbf{1}\mathbf{1}^{\mathrm{T}}}{n}\big)^{-1}}^2\\
	& +\frac{\big({k_3}+\frac{1}{k_2} -1 \big)(\sigma+k_5)}{2(	1-\frac{2}{k_3})\underline{\lambda}(\beta\mathcal{L}+\frac{\mathbf{1}\mathbf{1}^{\mathrm{T}}}{n}) }\big \lVert \beta\mathcal{L}\big( e_k-e_{k+1}\big)  \big\rVert^2\\	
	&	+\Big\langle e_{k+1} , z_{k+1}-z^*\Big\rangle-\Big\langle x_{k+1} -x^*, \beta\mathcal{L}\big(e_k-e_{k+1}\big)\Big\rangle\\
	\geq &  \frac{\sigma+k_5+1}{2} \Big(\big\lVert x_{k+1}-x^* \big\rVert_{P+k_4Q}^2  \\
	& +  \big\lVert z^*-z_{k+1}\big\rVert_{\big(\beta\mathcal{L}+\frac{\mathbf{1}\mathbf{1}^{\mathrm{T}}}{n}\big)^{-1}}^2 \Big).
	\end{split}
	\end{equation*}
	By monotonicity of $ E_k $ and the inequality 
	\begin{equation*}
	\Big \langle u,v \Big\rangle \leq \frac{k_5}{2} \big\lVert u \big\rVert_O^2+\frac{1}{2k_5} \big\lVert v\big\rVert_{O^{-1}}^2, \forall u,v\in\mathbb{R}^n,O\succ 0,
	\end{equation*}
	we are able to separate triggering errors from the primal-dual residual and arrive at \eqref{linear convergence}. This completes the proof.
\end{proof}





%

\section*{Acknowledgment}
	The authors would like to thank the Associate Editor and the anonymous reviewers for their constructive suggestions that have helped improve the paper.

\end{document}